\documentclass{amsart}
\usepackage[all]{xy}
\usepackage{enumerate,a4,amsfonts,amsmath,amsfonts,amssymb,amscd,amsthm}

\usepackage{mathrsfs}
\usepackage{hyperref}
\usepackage{color}

\usepackage[foot]{amsaddr}

\newcommand{\remove}[1]{} 

\newcommand{\C}{\mathbb{C}}

\newcommand{\F}{\mathbb{F}}

\newcommand{\N}{\mathbb{N}}
\newcommand{\Q}{\mathbb{Q}}

\newcommand{\Z}{\mathbb{Z}}

\newcommand{\calI}{{\mathcal{I}}}
\newcommand{\calJ}{{\mathcal{J}}}

\newcommand{\calL}{{\mathcal{L}}}

\newcommand{\calQ}{{\mathcal{Q}}}
\newcommand{\calR}{{\mathcal{R}}}
\newcommand{\calS}{{\mathcal{S}}}

\newcommand{\catC}{{\mathscr{C}}}
\newcommand{\catD}{{\mathscr{D}}}

\newcommand{\catP}{{\mathscr{P}}}

\newcommand{\fraka}{{\mathfrak{a}}}

\newcommand{\frakp}{{\mathfrak{p}}}

\newcommand{\SMatII}[4]{\left[\begin{array}{cc} {#1} & {#2} \\ {#3} &
{#4} \end{array}\right]}
\newcommand{\smallSMatII}[4]{\left[\begin{smallmatrix} {#1} & {#2} \\ {#3} &
{#4} \end{smallmatrix}\right]}

\newcommand{\suchthat}{\,:\,}
\newcommand{\where}{\,|\,}

\newcommand{\quo}[1]{\overline{#1}}

\newcommand{\veps}{\varepsilon}

\newcommand{\Trings}[1]{\left< #1 \right>}

\DeclareMathOperator{\Cent}{Cent}
\DeclareMathOperator{\Char}{char}

\DeclareMathOperator{\End}{End} %
\DeclareMathOperator{\gen}{gen} %

\DeclareMathOperator{\Hom}{Hom} %
\DeclareMathOperator{\id}{id}
\DeclareMathOperator{\im}{im} %
\DeclareMathOperator{\ind}{ind}
\DeclareMathOperator{\Jac}{Jac} %
\DeclareMathOperator{\Mor}{Mor} %
\newcommand{\op}{\mathrm{op}}

\DeclareMathOperator{\rank}{rank}

\DeclareMathOperator{\Spec}{Spec}

\newtheorem{thm}{Theorem}[section]
\newtheorem*{thm*}{Theorem}
\newtheorem{lem}[thm]{Lemma}
\newtheorem{prp}[thm]{Proposition}
\newtheorem{cor}[thm]{Corollary}

\newtheoremstyle{roman} 
    {8.0pt plus 2.0pt minus 4.0pt}                    
    {8.0pt plus 2.0pt minus 4.0pt}                    
    {\normalfont}                
    {}                           
    {\bfseries}                  
    {.}                          
    {5pt plus 1pt minus 1pt}     
    {}  

\theoremstyle{roman}

\newtheorem{notation}[thm]{Notation}
\newtheorem{example}[thm]{Example}
\newtheorem{remark}[thm]{Remark}

\theoremstyle{plain}

\numberwithin{equation}{section}

\newcommand{\units}[1]{{#1^\times}}   

\newcommand{\Herm}[2][]{\mathrm{UH}^{#1}(#2)}     
\newcommand{\Quad}[2][]{\mathrm{UQ}^{#1}(#2)}     
\newcommand{\scalarExt}[2]{\calR_{{#1}/{#2}}}     

\newcommand{\TDA}[2][]{{\mathrm{A\widetilde{r}}}_{2#1}({#2})}       

\newcommand{\rproj}[1]{\catP(#1)}                   

\newcommand{\Comm}[1]{\mathrm{Comm}\textrm{-}{#1}}  

\newcommand{\dual}[1]{{#1}^\vee}      

\newcommand{\nMat}[2]{\mathrm{M}_{#2}(#1)}    
\newcommand{\rdnorm}[1][]{\mathrm{Nrd}_{#1}}  

\newcommand{\Iso}{\mathrm{Iso}}
\newcommand{\uIso}{{\mathbf{Iso}}}
\newcommand{\uO}{{\mathbf{O}}}

\title{Patching and Weak Approximation in Isometry Groups}

\author{Eva Bayer-Fluckiger$^1$}
\author{Uriya A.\ First$^2$}
\address{$^1$\'{E}cole Polytechnique F\'{e}d\'{e}rale de Lausanne, Switzerland.}
\address{$^2$University of British Columbia, Canada.}

\date{\today}

\thanks{
The second named author has preformed the research in EPFL, the Hebrew University of Jerusalem and the University
of British Columbia (in this order), where he was supported by an SNFS grant \#IZK0Z2\_151061,
an ERC grant \#226135, and the UBC Mathematics Department, respectively.
}



\begin{document}

\maketitle

\begin{abstract}
    Let $R$ be a semilocal principal ideal domain. Two algebraic objects over $R$ in which
    scalar extension makes sense  (e.g.\ quadratic spaces)
    are said to be of the same \emph{genus} if they become isomorphic
    after extending scalars to all completions of $R$ and its fraction field.
    We prove that the number of isomorphism
    classes in the genus of unimodular quadratic spaces over (not necessarily commutative) \emph{$R$-orders} is always
    a finite power of $2$, and under further assumptions, e.g.\ that the order is hereditary, this number is $1$.
    The same result is also shown for related
    objects, e.g.\ systems of sesquilinear forms.
    A key ingredient in the proof is a weak approximation theorem for groups of isometries,
    which is valid over any (topological) base field, and even over semilocal base rings.
\end{abstract}

\setcounter{section}{-1}

\section{Introduction}
\label{section:intro}

    Let $R$ be a semilocal principal ideal domain, or equivalently, a Dedekind domain with
    finitely many maximal ideals. For $\frakp\in\Spec R$, let $R_\frakp$ denote the localization of
    $R$ at $\frakp$, and let $\hat{R}_{\frakp}$ denote the $\frakp$-adic completion of $R_{\frakp}$.
    Note that $F:=\hat{R}_0$ is just the fraction field of $R$.

    We define the \emph{genus} of a quadratic form $q$ over $R$ to be the set of isomorphism
    classes of
    quadratic forms that become isomorphic to $q$ over $\hat{R}_\frakp$ for all $\frakp\in \Spec R$
    (including $\frakp=0$). This resembles the (much stronger) notion of genus of quadratic forms over the integers \cite[\S102A]{OMeara63edit2000}.
    A classical result  states that the genus of integral quadratic forms is finite (see \cite[Th.~102:8, Th.~103:4]{OMeara63edit2000}
    and also \cite[Th.~3.4, Th.~4.2]{BayKeaWil89} for generalizations).

    Our notion of genus clearly generalizes to other objects defined over $R$ for which there is a  notion
    of scalar extension. This paper is concerned with proving that the genus is finite for various types of objects
    of quadratic nature.
    Of particular interest are cases where the genus consists of a
    single isomorphism class, since then it is enough to check isomorphism over the completions
    $\{\hat{R}_{\frakp}\}_{\frakp\in \Spec R}$ in order to prove isomorphism over $R$. This can be regarded
    as algebraic
    patching, of the kind
    that requires certain factorizations in the automorphism group.
    Such patching problems were considered by various authors, especially
    for torsors of group schemes; see for instance \cite{Nisne84}
    and \cite[\S3]{HarHarKra09}, 
    to name just a few examples.

\medskip

    Let $A$ be an $R$-algebra admitting a \emph{unitary $R$-algebra} structure (all definitions
    are recalled in section~\ref{section:preliminaries}). Then one can consider the genus of
    quadratic spaces over $A$.
    Assume henceforth that  $A$ is finitely generated
    and torsion-free as an $R$-module,
    and
    let $(P,[f])$ be a \emph{unimodular} quadratic space over $A$. We show that:
    \begin{enumerate}
        \item[(1)] $|\gen(P,[f])|$ is a finite power of $2$.
        \item[(2)] If $A$ is a hereditary,
        then $|\gen(P,[f])|=1$.
    \end{enumerate}
    Recall that the algebra $A$ is hereditary if its one-sided ideals are projective. Notable examples of hereditary
    orders include \emph{maximal orders}.
    We also bound  the size of the genus in the non-hereditary case (see Theorem~\ref{TH:main}).

    When $2\in\units{R}$, we extend the previous results to systems of sesquilinear forms
    and \emph{non-unimodular} hermitian forms, using  results from \cite{BayerFain96},
    \cite{BayerMold12} and \cite{BayFiMol13}. Specifically, we show that:
    \begin{enumerate}
        \item[(3)] Let $\{\sigma_i\}_{i\in I}$ be a family of $R$-involutions on $A$
        and let $(P,\{f_i\}_{i\in I})$ be a system of sesquilinear forms
        over $(A,\{\sigma_i\}_{i\in I})$. Then $|\gen(P,\{f_i\})|$ is a finite power of $2$.
        \item[(4)]
        Let $\sigma:A\to A$ be an $R$-involution and let $u\in\Cent(A)$  be an element
        satisfying $u^\sigma u=1$. Assume that $A$ is hereditary.
        Then $|\gen(P,f)|=1$ for every $u$-hermitian space $(P,f)$ over $(A,\sigma)$ (unimodularity is not assumed).
    \end{enumerate}
    As an application of (4), we show that Witt's Cancellation Theorem and a variant of Springer's
    Theorem hold for hermitian forms over involutary hereditary orders (unimodularity is not assumed).

\medskip

    A main tool in the proofs, which may be of interest in its own right,
    is a weak approximation theorem: Let $K$ be a topological commutative  \emph{semilocal} ring,
    let $F$ be a dense subring,  let $A$ be a \emph{unitary $F$-algebra}, and let $(P,[f])$
    be a unimodular quadratic space over $A$. Denote by $O([f])$ the group of isometries
    of $[f]$, and let $[f_K]$ denote
    the scalar extension of $[f]$ to $A\otimes_F K$.
    Then, under mild assumptions, the closure of $O([f])$ in $O([f_K])$ is of finite
    index. Furthermore,
    if $F$ is a field, $K$ is a product of fields, and and $O^+$ denotes
    the connected component of $O([f])$ (when viewed as a group scheme),
    then $\overline{O^+(F)}=O^+(K)$.

    Weak approximation theorems for adjoint algebraic groups over
    arbitrary topological fields were studied previously; see \cite{Thang96}
    and references therein (for instance). Our approach is somewhat different and relies
    on generalizations of Witt's Theorem to quadratic spaces over semilocal rings (\cite{Reiter75}, \cite{Fi14A}).

    We  note that in our general setting $O([f])$ can indeed be regarded as an affine scheme over $\Spec F$, which is also
    smooth and faithfully flat. However,
    when $F$ is replaced with a commutative ring,
    some mild assumptions are necessary.
    For the sake of completeness, this is
    discussed in an appendix. The appendix  is based on  communications with
    Mathieu Huruguen, and we thank him for his contribution.

\medskip

    The paper is organized as follows: Section~\ref{section:preliminaries} recalls
    unitary rings and quadratic spaces. Section~\ref{section:topology} deals with some topological
    issues required to phrase the weak approximation theorem, which is proved in section~\ref{section:density}.
    In sections~\ref{section:double-cosets} and~\ref{section:patching}, we prove the finiteness of the genus,
    and in section~\ref{section:size-one} we characterize various cases in which the size of the genus is $1$.
    Finally, section~\ref{section:further-quadratic-objects} extends the previous results
    to systems of sesquilinear forms and non-unimodular hermitian forms. The cancellation and the variant
    of Springer's Theorem mentioned above are also proved there.
    In the appendix, we show that  isometry groups can be regarded
    as smooth affine group schemes, provided mild assumptions.

\section{Preliminaries}
\label{section:preliminaries}

    This section recalls  hermitian and quadratic forms over rings and various related notions.
    See  \cite{Bass73AlgebraicKThyIII} and
    \cite{Kn91}
    for an extensive discussion.

\subsection{Hermitian Forms}
\label{subsection:herm-forms}

    Let $(A,\sigma)$ be a ring with involution and let $u\in \Cent(A)$ be an element satisfying $u^\sigma u=1$.
    Denote by $\rproj{A}$ the category of finitely generated projective right $A$-modules.
    For  $P\in\rproj{A}$, we make $P^*:=\Hom_A(P,A)$ into a right $A$-module
    by setting
    \[(\psi a)x=a^\sigma(\psi x)\qquad \forall\,\psi\in P^*,\,a\in A,\,x\in P\ .\]
    Observe that $*:\rproj{A}\to\rproj{A}$ is a contravariant functor. For every morphism
    $f\in\Hom_A(P,Q)$, the dual $f^*\in\Hom_A(Q^*,P^*)$ is given by $f^*\psi=\psi \circ f$
    ($\psi\in P^*$).

    Every morphism $f\in\Hom_A(P,P^*)$ gives rise to a map $\tilde{f}:P\times P\to A$
    given by
    \[\tilde{f}(x,y)=(fx)y\qquad\forall x,y\in P\ .\]
    The map $\tilde{f}$ is biadditive and satisfies
    \begin{equation}\label{EQ:sesquilinear-dfn}
    \tilde{f}(xa,yb)=a^\sigma \tilde{f}(x,y)b\qquad\forall\, x,y\in P,\, a,b\in A\ .
    \end{equation}
    Conversely, it is easy to see that any biadditive map $\tilde{f}:P\times P\to A$ satisfying \eqref{EQ:sesquilinear-dfn}
    is induced by a unique homomorphism $f\in\Hom_A(P,P^*)$. The map $\tilde{f}$ is called a \emph{sesquilinear
    form} and the pair $(P,f)$ or $(P,\tilde{f})$ is called a \emph{sesquilinear space} (over $(A,\sigma)$).
    We say that $(P,f)$ is \emph{unimodular} if $f$ is an isomorphism.

    There is a natural homomorphism $\omega_P:P\to P^{**}$ given by
    \[
    (\omega_P x)\phi=(\phi x)^\sigma u\qquad\forall\, x\in P,\,\phi\in P^*\ .
    \]
    It is well-known that $\omega_P$ is an isomorphism (when $P\in\rproj{A}$). Notice
    that $\omega_P$ depends on $u$.
    A \emph{$u$-hermitian space} over $(A,\sigma)$ is a sesquilinear space $(P,f)$ such that $f=f^*\omega_P$.
    This is equivalent to
    \[
    \tilde{f}(x,y)=\tilde{f}(y,x)^\sigma u\qquad\forall\, x,y\in P\ .
    \]

    Let $(P,f)$ and $(P',f')$ be sesquilinear spaces. An isometry
    from $(P,f)$ to $(P',f')$ is an isomorphism $\phi:P\to P'$ such that
    $f=\phi^*f'\phi$. The latter is equivalent to
    \[
    \tilde{f}(x,y)=\tilde{f}'(\phi x,\phi y)\qquad\forall\, x,y\in P\ .
    \]
    The group of isometries of $(P,f)$ is denoted by $O(f)$.

    Orthogonal sums of sesquilinear forms are defined in the usual way.
    We denote by $\Herm[u]{A,\sigma}$
    the category of unimodular $u$-hermitian forms
    over $(A,\sigma)$.

\subsection{Quadratic Forms}
\label{subsection:quadratic-forms}

    Keep the setting of \ref{subsection:herm-forms}.
    To define quadratic spaces, additional data is needed.
    Set
    \[
    \Lambda^{\min}(u)=\{a-a^\sigma u\where a\in A\} \qquad\text{and}\qquad \Lambda^{\max}(u)= \{a\in A\suchthat a^\sigma u=-a\}
    \]
    A \emph{form parameter}
    (for $(A,\sigma,u)$) consists of an additive group $\Lambda$ such that
    \[\Lambda^{\min}(u)\subseteq\Lambda\subseteq\Lambda^{\max}(u)\qquad \text{and}\qquad
    a^\sigma\Lambda a\subseteq\Lambda\qquad\forall\, a\in A\ .
    \]
    In this case, the quartet $(A,\sigma,u,\Lambda)$ is called a \emph{unitary ring}. (It is also
    common to call the pair $(u,\Lambda)$ a form parameter.)
    When $2\in \units{A}$, $\Lambda^{\min}(u)=\Lambda^{\max}(u)$ because
    any $a\in \Lambda^{\max}(u)$ satisfies $a=\frac{1}{2}a-(\frac{1}{2}a)^\sigma u\in\Lambda^{\min}(u)$,
    so there is only one possible form parameter.

    For every $P\in\rproj{A}$,
    define
    \[
    \Lambda_P=\{f\in\Hom_A(P,P^*)\suchthat f=-f^*\omega_P ~\text{and}~\tilde{f}(x,x)\in\Lambda~\text{for all}~x\in P\}\ .
    \]
    Given $f\in \Hom_A(P,P^*)$, denote by $[f]$ the class of $f$ in $\Hom_A(P,P^*)/\Lambda_P$.
    A \emph{quadratic space} (over $(A,\sigma,u,\Lambda)$) is a pair $(P,[f])$ with $P\in\rproj{A}$
    and $f\in\Hom_A(P,P^*)$. Associated with $[f]$ are the $u$-hermitian form
    \[h_f=f+f^*\omega_P\]
    and the
    quadratic map $\hat{f}:P\to A/\Lambda$
    given by
    \[\hat{f}(x)=\tilde{f}(x,x)+\Lambda\ .\]
    Both $h_f$ and $\hat{f}$ are determined by the class $[f]$ (rather than $f$), and conversely, they also determine $[f]$.
    We say that $(P,[f])$ is \emph{unimodular} if $h_f:P\to P^*$ is bijective.

    Let $(P',[f'])$ be another quadratic space. An isometry
    from $(P,[f])$ to $(P',[f'])$ is an isomorphism $\phi:P\to P'$
    such that
    \begin{equation}\label{EQ:isometry-for-quad-forms}
    [\phi^* f'\phi]=[f]
    \end{equation}
    in $\Hom_A(P,P^*)/\Lambda_P$.
    This is equivalent to
    \begin{equation}\label{EQ:isometry-equiv-cond}
    h_{f'}(\phi x,\phi y)=h_f(x,y)\qquad\text{and}\qquad \hat{f}'(\phi x)=\hat{f}(x)\qquad\forall x,y\in P\ .
    \end{equation}
    We let $O([f])$  denote the  isometry group of $(P,[f])$.

    The category of unimodular  quadratic spaces over $(A,\sigma,u,\Lambda)$
    is denoted by $\Quad[u,\Lambda]{A,\sigma}$.

    \begin{remark}\label{RM:two-is-invertible}
        When $2\in \units{A}$, $[f]$
        can be recovered from $h:=h_f$ via $[f]=[\frac{1}{2}h]$.
        Equation \eqref{EQ:isometry-for-quad-forms}  is therefore equivalent to $\phi^*h'\phi-h\in\Lambda_P$
        where $h'=h_{f'}$. Write $g=\phi^*h'\phi-h$. Then $g=g^*\omega $ (since $h=h^*\omega$ and $h'=h'^*\omega$) and
        $g=-g^*\omega$ (since $g\in\Lambda_P$), hence $g=0$ (because $2\in\units{A}$). Therefore,  $\phi$
        is an isometry from $(P,[f])$ to $(P',[f'])$ if and only if it is an isometry
        from $(P,h)$ to $(P',h')$. It follows that $\Quad[u,\Lambda]{A,\sigma}\cong\Herm[u]{A,\sigma}$
        when $2\in \units{A}$.
    \end{remark}

\subsection{Scalar Extension}
\label{subsection:unitary-algs}

    Let $F$ be a commutative ring. Throughout, all tensor products are taken over $F$.
    A \emph{unitary $F$-algebra} is a unitary
    ring $(A,\sigma,u,\Lambda)$ such that $A$ is an $F$-algebra, $\sigma$ is $F$-linear,
    and $\Lambda$ is an $F$-submodule of $A$. For a commutative ring extension $K/F$, define
    \[
    \scalarExt{K}{F}(A,\sigma,u,\Lambda)=(A\otimes K,\sigma\otimes\id_K,u\otimes1_K,\Lambda\otimes_F^A K)\ .
    \]
    Here,  $\Lambda\otimes_F^A K$  denotes the image of $\Lambda\otimes K$ in $A\otimes K$
    (when $K$ is flat over $F$, the distinction between $\Lambda\otimes K$
    and $\Lambda\otimes_F^A K$ is unnecessary).
    It is easy to see that $(B,\tau,v,\Gamma):=\scalarExt{K}{F}(A,\sigma,u,\Lambda)$ is a unitary ring.

    For every $P,Q\in \rproj{A}$ and $\phi\in\Hom_A(P,Q)$, let
    \[
    P_K=P\otimes K\qquad\text{and}\qquad \phi_K=\phi\otimes \id_K\in\Hom_{A_K}(P_K,Q_K)\ .
    \]
    (Here, $P_K$ is considered as an $A_K$-module by setting $(x\otimes k)(a\otimes m)=(xa)\otimes (km)$
    for all $x\in P$, $a\in A$, $k,m\in K$.) The assignment $P\mapsto P_K:\rproj{A}\to\rproj{A_K}$
    is a functor denoted by $\scalarExt{K}{F}$.

    For every sesquilinear form $f$ on $P$, define a sesquilinear form $f_K$ on $P_K$ by
    linearly extending
    \[
    \tilde{f}_K(x\otimes k,x'\otimes k')=\tilde{f}(x,x')\otimes kk'\qquad\forall\, x,x'\in P,\, k,k'\in K\ .
    \]
    It is  easy to check that  $f\in \Lambda_P$ implies $f_K\in \Gamma_{P_K}$, so
    the map sending $(P,[f])$ to $(P_K,[f_K])$ is well-defined.
    The quadratic space (resp.\ sesquilinear space)  $(P_K,[f_K])$ (resp.\ $(P_K,f_K)$)
    is called the \emph{scalar extension} of $(P,[f])$ (resp.\  $(P,f)$).
    This gives rise
    to functors
    \begin{align*}
    \Herm[u]{A,\sigma}&\to\Herm[v]{B,\tau},\\
    \Quad[u,\Lambda]{A,\sigma}&\to \Quad[v,\Gamma]{B,\tau},
    \end{align*}
    which, by abuse of notation, are all denoted  $\scalarExt{K}{F}$.
    (The action of $\scalarExt{K}{F}$ on isometries is the same as its action on morphisms of $\rproj{A}$.)

\medskip

    We now give another description of $f_K$ and $[f_K]$, which does not pass through $\tilde{f}$
    and will be useful later.

    \begin{lem}\label{LM:natural-isomorphism}
        For all $P,P'\in \rproj{A}$, there are \emph{natural} isomorphisms
        \[
        \Hom_A(P,P')_K\cong \Hom_{A_K}(P_K,P'_K)
        \qquad\text{and}\qquad
        (P^*)_K\cong (P_K)^*\ ,
        \]
        where the latter is an isomorphism of $A_K$-modules.
    \end{lem}

    \begin{proof}
        Since $P^*=\Hom_A(P,A)$ and $(P_K)^*=\Hom_{A_K}(P_K,A_K)$, the second isomorphism
        is just a special case of the first isomorphism (and it is straightforward to check
        that we in fact obtain  an isomorphism of $A_K$-modules).
        Define $\Phi:\Hom_A(P,P')_K\to \Hom_{A_K}(P_K,P'_K)$
        by
        \[
        (\Phi(\phi\otimes k))(x\otimes m)=\phi x\otimes km\qquad\forall\, \phi\in\Hom_A(P,P')\, ,
        x\in P,\, k,m\in K\ .
        \]
        It is routine to verify that $\Phi$ is natural (in the categorical sense) and   an isomorphism
        when $P=P'=A_A$. The naturality of $\Phi$ now implies that it is also an isomorphism
        when $P$ and $P'$ are summands of f.g.~free modules.
    \end{proof}

    \begin{prp}\label{PR:scalar-ext-for-quad-forms}
        Every  $P\in\rproj{A}$ gives rise to  a commutative diagram with exact rows
        and such that $\alpha$ is onto and $\beta$ and $\gamma$ are isomorphisms.
        \[
        \xymatrix{
        & (\Lambda_{P})_K\ar[r] \ar[d]^\alpha & \Hom_A(P,P^*)_K \ar[r] \ar[d]^\beta & (\Hom_A(P,P^*)/\Lambda_P)_K \ar[r] \ar[d]^\gamma & 0\\
        0 \ar[r] & \Gamma_{P_K} \ar[r] & \Hom_{A_K}(P_K,(P_K)^*) \ar[r] & \Hom_{A_K}
        (P_K,(P_K)^*)/\Gamma_{P_K}  \ar[r] & 0
        }
        \]
    \end{prp}

    \begin{proof}
        The exactness of the bottom row holds by definition, and the exactness of the top row
        follows from the exact sequence
        \[0\to \Lambda_P\to\Hom_A(P,P^*)\to \Hom(P,P^*)/\Lambda_P\to 0\]
        by tensoring with $K$.
        By Lemma~\ref{LM:natural-isomorphism}, we have a natural isomorphism
        \[
        \Hom_A(P,P^*)_K\cong \Hom_{A_K}(P_K,(P^*)_K)\cong \Hom_{A_K}(P_K,(P_K)^*)
        \]
        which we take to be $\beta$. Explicitly, for all $f\in\Hom_A(P,P^*)$ and $k\in K$, one has
        \[
        (\beta (f\otimes k))\!\tilde{\phantom{f}}(x\otimes a, y\otimes b)=\tilde{f}(x,y)\otimes kab\qquad\forall\, x,y\in P,\, a,b\in K\ .
        \]
        This is easily seen to imply that the image of $(\Lambda_P)_K$ in $\Hom_A(P,P^*)_K$ is mapped
        by $\beta$ into $\Gamma_{P_K}$. We define $\alpha:(\Lambda_P)_K\to \Gamma_{P_K}$ to be this map.
        The map $\gamma$ is now induced by $\alpha$ and $\beta$ in the standard way, namely,
        $\gamma([f]\otimes k)=[\beta (f\otimes k)]$.
        The five lemma implies that $\gamma$ is an isomorphism if $\alpha$ is onto, which is what we shall verify.

        Assume first that $P$ is a free $A$-module with basis $\{x_i\}_{i=1}^t$.
        Then $\{y_i\}:=\{x_i\otimes 1\}$ is a basis of $A_K$. A morphism $f\in\Hom_{A_K}(P_K,(P_K)^*)$
        is completely determined by the values $\{\tilde{f}(y_i,y_j)\}_{i,j}$,
        and it belongs to $\Gamma_{P_K}$ if and only if $\tilde{f}(y_i,y_i)\in \Gamma$ and
        $\tilde{f}(y_i,y_j)+\tilde{f}(y_j,y_i)^{(\sigma\otimes \id)} (u\otimes 1)=0$ for all $i,j$.
        For such $f$ and $i\leq j$, write $\tilde{f}(y_i,y_j)=\sum_{s}a_{ij}^{(s)}\otimes k_{ij}^{(s)}$
        with $\{a_{ij}^{(s)}\}\subseteq A$, $\{k_{ij}^{(s)}\}\in K$. Since $\Gamma=\Lambda\otimes^A_F K$, we can
        choose $a_{ii}^{(s)}$ to be in $\Lambda$ for all $i,s$. Now, for all $i\leq j$ and $s$, let $g_{ij}^{(s)}$ denote the unique morphism
        in $\Hom_A(P,P^*)$ satisfying
        \[\tilde{g}_{ij}^{(s)}(x_n,x_m)=\left\{
        \begin{array}{ll}
        a_{ij}^{(s)} & (n,m)=(i,j) \\
        -(a_{ij}^{(s)})^\sigma u & (n,m)=(j,i) \\
        0 & \text{otherwise}
        \end{array}\right.
        \]
        It is  routine to verify that $g_{ij}^{(s)}\in\Lambda_P$ and $\alpha(\sum_{i\leq j}\sum_s g_{ij}^{(s)}\otimes k_{ij}^{(s)})=f$,
        as required.

        For general $P$, choose $P'\in\rproj{A}$  such that $P\oplus P'$ is  free. If $f\in \Gamma_{P_K}$,
        then $f\oplus 0\in \Gamma_{P_K\oplus P'_K}$, hence we can write $f\oplus 0=\alpha(\sum_i g_i\otimes k_i)$ for
        $\{g_i\}_{i=1}^r\subseteq \Lambda_{P\oplus P'}$ and $\{k_i\}_{i=1}^r\subseteq K$.
        For all $i$, define $h_i\in \Hom_A(P,P^*)$ by $\tilde{h}_i(x,y)=\tilde{g}_i(x\oplus 0,y\oplus 0)$.
        Then  $\{h_i\}_{i=1}^r\subseteq\Lambda_P$ and an easy computation shows that $\alpha(\sum_i h_i\otimes k_i)=f$.
    \end{proof}

    Let $P\in\rproj{A}$ and $f\in\Hom_A(P,P^*)$. It is straightforward to check
    that
    \[
    f_K=\beta(f\otimes 1_K)\qquad\text{and}\qquad [f_K]=\gamma([f]\otimes 1_K)
    \]
    where $\beta$ and $\gamma$ are as in Proposition~\ref{PR:scalar-ext-for-quad-forms}.

\subsection{Orthogonal Unitary Rings}
\label{subsectnio:csa-s}

    Recall that a ring with involution $(A,\sigma)$ is  \emph{simple} if $A$ admits no nontrivial two-sided
    ideals
    $I$ satisfying $I^\sigma=I$. In this case, it is well-known that $A$ is either simple,
    or $A\cong B\oplus B^\op$ with $B$ a simple ring, and $\sigma$ exchanges $B$ and $B^\op$.
    If $A$ also happens to be artinian, then the Artin-Wedderburn Theorem implies that
    $A\cong \nMat{D}{n}$ where $D$ is a division ring or a product of a division ring and its opposite.

    We call a unitary ring $(A,\sigma,u,\Lambda)$ \emph{simple} if
    $(A,\sigma)$ is simple as a ring with involution.

    \begin{prp}\label{PR:factorization-of-unitary-rings}
        Let $(A,\sigma,u,\Lambda)$ be a unitary ring such that $A$ is a semisimple
        ring. Then $(A,\sigma,u,\Lambda)$ factors into a product of unitary rings
        \[
        (A,\sigma,u,\Lambda)\cong \prod_{i=1}^t(A_i,\sigma_i,u_i,\Lambda_i):=\Big(\prod_iA_i,\prod_i\sigma,(u_i)_i,\prod_i\Lambda_i\Big)
        \]
        such that each $(A_i,\sigma_i,u_i,\Lambda_i)$ is simple artinian.
    \end{prp}

    \begin{proof}
        This is well-known; see for instance \cite[Pr.~2.7]{Fi14A}.
    \end{proof}

    We now recall a notion of orthogonality for simple artinian unitary rings
    defined in \cite[\S2.4]{Fi14A} (see also the orthogonality defined in \cite[Ch.~4,~\S2]{Bass73AlgebraicKThyIII}).

\medskip

    Let $A$ be a central simple algebra over a field $K$ (see for instance \cite[Ch.~I]{InvBook}).
    The \emph{degree} and \emph{index} of $A$ are denoted by $\deg A$ and $\ind A$, respectively.
    Recall that  involutions of the first kind on $A$ (i.e.\ involutions fixing $K$ point-wise)
    divide into two families: \emph{orthogonal} and \emph{symplectic}
    (cf.~\cite[\S{}I.2]{InvBook}). Recall also that if $\Char K\neq 2$, then $\sigma$ is orthogonal if and only if
    $\dim_K\{a-a^\sigma\where a\in A\}=\frac{1}{2}n(n-1)$ where $n=\deg A$.

\medskip

    A simple artinian unitary ring $(A,\sigma,u,\Lambda)$ is called
    \emph{orthogonal} if:
    \begin{enumerate}
        \item[(1)] $A$ is simple as a ring and finite dimensional over its center $K$ (which is a field in this case),
        \item[(2)] $\sigma$ is of the first kind (i.e.\ $\sigma|_K=\id_K$),
        \item[(3)] one of the following holds:
        \begin{enumerate}
            \item[(3a)] $\Char K\neq 2$, $\sigma$ is orthogonal and $u=1$,
            \item[(3b)] $\Char K\neq 2$, $\sigma$ is symplectic and $u=-1$,
            \item[(3c)] $\Char K=2$ and $\Lambda=\Lambda^{\min}(u)$.
        \end{enumerate}
        (These conditions are equivalent to $\Lambda$ being a $K$-vector space
        and satisfying $\dim_K\Lambda=\frac{1}{2}n(n-1)$, where $n=\deg A$; see
        \cite[\S2.4]{Fi14A}.)
    \end{enumerate}
    If in addition $A$ is split (as a central simple $K$-algebra),  we say
    that $(A,\sigma,u,\Lambda)$ is \emph{split-orthogonal}.

\subsection{Transfer}
\label{subsection:transfer}

    We now recall the method of \emph{transfer into the endomorphism ring}.
    This is a special case of transfer in \emph{hermitian categories}.
    See \cite[\S{}II.3]{Kn91} or \cite[Pr.~2.4]{QuSchSch79} for further details.

\medskip

    Let $(A,\sigma,u,\Lambda)$
    be a unitary ring.
    Fix a unimodular $u$-hermitian form $(Q,h)\in\Herm[u]{A,\sigma}$
    and let $B=\End_A(Q)$.
    The form $h$ induces an involution $\tau=\tau(h):B\to B$
    given by $\phi^\tau=h^{-1} \phi^*  h$. Equivalently, $\phi^\tau$ is the unique
    element of $B$ satisfying $\tilde{h}(\phi x,y)=\tilde{h}(x,\phi^\tau y)$ for all $x,y\in Q$.
    We further define $\Gamma=\Gamma(h,\Lambda)=h^{-1}\Lambda_Q$. It straightforward to
    check that $(B,\tau,1_B,\Gamma)$ is a unitary ring.
    If $(P,f)$ is a sesquilinear form, we define
    a sesquilinear form $(B_B,\mathrm{T}_hf)$ by
    \[
    \widetilde{\mathrm{T}_hf}(b,b')=b^\tau  (h^{-1}f) b'\ .
    \]
    It is easy to see that $f\in\Lambda_Q$ implies $\mathrm{T}_hf\in \Gamma_B$,
    hence the map $[f]\mapsto [\mathrm{T}_hf]$ is well-defined.

    \begin{prp}[Transfer] \label{PR:transfer}
        Keep the previous setting and identify $\End_B(B_B)$ with $B=\End_A(Q)$
        via $\psi\mapsto \psi(1_B)$. Then:
        \begin{enumerate}
            \item[{\rm(i)}] $[f]$ is unimodular if and only if $[\mathrm{T}_hf]$ is unimodular.
            \item[{\rm(ii)}] $O([f])=O([\mathrm{T}_hf])$.
            \item[{\rm(iii)}] If $(P,f')$ is another sesquilinear space, then $[f]\cong [f']$
            if and only if $[\mathrm{T}_hf]\cong [\mathrm{T}_hf']$.
        \end{enumerate}
        More generally, there is an isomorphism between the category of quadratic spaces
        over $(A,\sigma,u,\Lambda)$
        with base module $P$ and the category of quadratic spaces over $(B,\tau,1,\Gamma)$
        with base module $B_B$.
    \end{prp}

    \begin{proof}
        This is routine. See  \cite[\S{}II.3]{Kn91} for a proof in a more general setting.
    \end{proof}

    \begin{remark}\label{RM:free-base-module}
        (i)
        If $(P,[f])\in \Quad[u,\Lambda]{A,\sigma}$, then by applying Proposition~\ref{PR:transfer}
        with $(Q,h)=(P,h_f)$, we may transfer certain statements about quadratic
        spaces with base module $P$ to analogous statements
        about quadratic spaces with base module $B_B$ ($B=\End_A(P)$).
        This allows us to assume that the base module is   isomorphic to the base ring, and free
        in particular.

        (ii) Transfer is compatible with scalar extension in the
        sense of~\ref{subsection:unitary-algs}. This follows from Lemma~\ref{LM:natural-isomorphism}
        and left as an exercise to the reader; see also~\cite[\S2E]{BayFiMol13}.
    \end{remark}

    \begin{prp}\label{PR:transfer-preserve-split-orth}
        Keep the previous setting and assume $Q\neq 0$. If $(A,\sigma,u,\Lambda)$ is simple artinian,
        then  $(B,\tau,1,\Gamma)$ is simple artinian. In this case,
        $(A,\sigma,u,\Lambda)$ is split-orthogonal if and only if
        $(B,\tau,1,\Gamma)$ is split-orthogonal.
    \end{prp}

    \begin{proof}
        If $A$ is simple artinian as a ring, then it is well-known that
        $B=\End_A(Q)$ is also a simple artinian ring. If $A$ is not simple as a ring,
        there exists an idempotent $e\in \Cent(A)$ with $e^\sigma e=0$ and $e^\sigma+e=1$,
        and $eAe$, $(1-e)A(1-e)$ are simple artinian rings.
        This implies $Q=Qe\oplus Q(1-e)$ and $B=\End_A(Q)\cong\End_{eAe}(Qe)\times \End_{(1-e)A(1-e)}(Q(1-e))$.
        In particular, $\End_A(Q)$ is a product of two simple artinian rings.
        Let $a=\id_{Qe}\oplus 0$ and $b=0\oplus \id_{Q(1-e)}$. It is easy
        to check that $\tilde{h}(ax,y)=\tilde{h}(x,by)$, and hence
        $a^\tau=b$. The only nontrivial ideals of $B$
        are $\End_{eAe}(Qe)\times 0$ and $0\times \End_{(1-e)A(1-e)}(Q(1-e))$. Since
        these ideals contain $a$ and $b$, respectively, it follows that $B$ has no non-trivial ideals
        invariant under $\tau$.

        Suppose now that $(A,\sigma,u,\Lambda)$ is split-orthogonal. Then $A\cong \nMat{K}{n}$ for a field
        $K$, and $\sigma$ is of the first kind. Using~\cite[Pr.~2.5, Ex.~2.10]{Fi14A},
        we may assume $\sigma$ is the transpose involution and $u=1$. Let $e$ be the matrix
        unit $e_{11}$. By~\cite[Pr.~2.4, Pr.~2.11]{Fi14A}, we may replace $(A,\sigma,u,\Lambda)$
        with $(eAe,\sigma|_{eAe},eu,e\Lambda e)$ and $P$ with $Pe$. That is,
        we may assume $(A,\sigma,u,\Lambda)=(K,\id_K,1,0)$.
        Now, if $m=\dim_KP$, then $B\cong \nMat{K}{m}$ and it is easy to check that $\Gamma$
        is a $K$-vector space of dimension $\frac{1}{2}m(m-1)$, so $(B,\tau,1,\Gamma)$ is split-orthogonal.

        Conversely, assume that $(B,\tau,1,\Gamma)$ is split-orthogonal. If $A$ is not simple artinian, then the argument
        above implies that the involution on $B$
        is of the second kind, which is impossible. Thus, $A$ is simple artinian. In fact,
        if $B\cong\nMat{K}{m}$ where $K$ is a field, then $A\cong \nMat{K}{n}$ for some $n\in\N$.
        Identifying $\Cent(A)$ and $\Cent(B)$ with $K$, it is easy to see that $\sigma|_K=\tau|_K$,
        hence  $\sigma$ is of the first kind.
        As in the previous paragraph, we may again assume $A=K$ and $\sigma=\id_K$.
        This implies $u\in\{\pm1\}$.
        If $\Char K\neq 2$ and $u=-1$, then $h$ is an alternating $K$-bilinear form, and so
        $\tau$ is symplectic, which is impossible.
        Thus, $u=1$.
        Now, take an $A$-basis $\{x_1,\dots,x_m\}$ of $P$
        with dual basis $\{\phi_1,\dots,\phi_m\}\subseteq P^*$, and let $f_{ij}=[x\mapsto x_i(\phi_j x)]\in\Hom_A(P,P^*)$.
        Then $\{f_{ij}\}_{1\leq i,j\leq m}$ is a $K$-basis of $\Hom_A(P,P^*)$
        and it is easy to check that $\Lambda_P=\sum_{i<j}(f_{ij}-f_{ji})K+\sum_{i}f_{ii}\Lambda$.
        Thus, $\Gamma=\sum_{i<j}h^{-1}(f_{ij}-f_{ij})K+\sum_ih^{-1}f_{ii}\Lambda$,
        while $\Gamma^{\min}(1)=\sum_{i<j}h^{-1}(f_{ij}-f_{ij})K$ (note that $(h^{-1}f_{ij})^\tau=h^{-1}f_{ji}$; use
        the identities $h=h^*\omega_P$ and $f_{ij}^*\omega_P=f_{ji}$).
        Since  $(B,\tau,1,\Gamma)$ is split-orthogonal,
        we  have $\Gamma=\Gamma^{\min}(1)$,  so $\Lambda=0$, as required.
    \end{proof}

\subsection{The Dickson Map}
\label{subsection:Dickson-inv}

    Let $(A,\sigma,u,\Lambda)$ be a \emph{split-orthogonal} simple artinian unitary ring,
    let $K=\Cent(A)$
    and let $(P,[f])\in\Quad[u,\Lambda]{A,\sigma}$. The \emph{Dickson map}
    (also called \emph{pseudodeterminant} or \emph{quasideterminant}) is a surjective group homomorphism
    \[
    \Delta=\Delta_{[f]}:O([f])\to \Z/2\Z\ .
    \]
    In case $2\in \units{A}$, it can be defined using the \emph{reduced norm} in $E:=\End_A(P)$ via
    \[
    \rdnorm[E/K](\psi)=(-1)^{\Delta(\psi)}\qquad\forall \psi\in O([f])\ ,
    \]
    and in general, it can be  defined by
    \[
    \Delta(\psi)=\frac{\dim_K(1-\psi)E}{\deg E} +2\Z \qquad \forall\, \psi\in O([f])\ .
    \]
    (This indeed yields a group homomorphism.)
    Viewing $O([f])$ as an affine group scheme over $K$
    (see the
    appendix) and $\Z/2\Z$ as a constant group scheme over $K$, $\Delta$ becomes a morphism
    of algebraic groups (defined over $K$), and the kernel of $\Delta$
    is the (Zariski) connected component of $O([f])$.
    See \cite[\S5.1]{Fi14A} and the references therein for proofs and further discussion.

\section{Some Topology}
\label{section:topology}

    In this section, we recall several facts  allowing us to properly topologize
    various algebraic objects.

\medskip

    Let $K$ be a commutative ring and let $\Comm{K}$ denote the category
    of commutative $K$-algebras.
    By
    a \emph{scheme over $K$} or a \emph{$K$-scheme} we mean a scheme $X$ which is of finite type over
    $\Spec K$.
    For every $L\in\Comm{K}$, let $X(L)$ denote the $L$-points of $X$ (i.e.\ the set of $K$-morphisms
    $\Spec L\to X$). The map $X\mapsto X(L):\Comm{K}\to \mathrm{Set}$ is a functor called
    the \emph{functor of points} of $X$.
    This functor is \emph{representable} precisely when $X$ is an affine scheme (over $K$).
    In this case, Yoneda's Lemma implies that the functor $L\mapsto X(L)$ determines $X$ up to isomorphism;
    see for instance \cite[\S2]{Water79IntroAffineGrpSchemes}.

    \begin{prp}\label{PR:topologizing-schemes}
        Let $K$ be a Hausdorff topological commutative ring. There is a unique way to topologize the
        set of $K$-points of all affine $K$-schemes $X$ such that:
        \begin{enumerate}
            \item[$(1)$] The assignment $X\mapsto X(K)$ is a functor from affine $K$-schemes to topological spaces,
            and it is compatible with fibered products.
            \item[$(2)$] Closed immersions become   closed  embeddings when restricted to $K$-points.
            \item[$(3)$] $\mathbb{A}^1_K(K)=K$ is given the topology of $K$.
        \end{enumerate}
    \end{prp}

    \begin{proof}
        See \cite[Pr.~2.1]{Conrad}.
        Explicitly, if $X=\Spec A$ for a commutative $K$-algebra $A$,
        then the topology on $X(K)\cong\Hom_{\Comm{K}}(A,K)$ is  the subspace topology
        induced from the product topology on $\Hom_{\mathrm{Set}}(A,K)=K^A$.
    \end{proof}

    \begin{example}
        It follows form the proposition that
        if $X$ is an affine scheme over $K$ and $X\to \mathbb{A}^n_K$ is a closed immersion,
        then the topology on $X(K)$ is induced from the product topology on $\mathbb{A}^n_K(K)=K^n$.
        The resulting topology is independent of the immersion $X\to \mathbb{A}^n_K$.
    \end{example}

    The following proposition allows us to realize finitely generated projective modules
    over $K$ as the $K$-points of affine schemes.
    Recall that $\rproj{K}$ denotes the category of  finitely generated projective right $K$-modules.

    \begin{prp}\label{PR:realizing-projective-mods}
        There is a functor $P\mapsto \underline{P}$ from $\rproj{K}$ to the category
        of affine  schemes over $K$
        such that:
        \begin{enumerate}
            \item[$(1)$] The functors $[L\mapsto P_L:=P\otimes_KL]$ and $[L\mapsto \underline{P}(L)]$
            from $\Comm{K}$ to $\mathrm{Set}$ are isomorphic.
            \item[$(2)$] If $Q$ is a summand of $P$, then the corresponding morphism
            $\underline{Q}\to \underline{P}$ is a closed immersion.
            \item[$(3)$] $\underline{P\oplus Q}$ is canonically isomorphic to $\underline{P}\times_{\Spec K} \underline{Q}$.
            \item[$(4)$] $\underline{K^n}$ is  isomorphic to $\mathbb{A}^n_K$.
        \end{enumerate}
    \end{prp}

    \begin{proof}
        This is well-known: Let $P^\vee:=\Hom_K(P,K)$,
        let $S^nP^\vee$ be the $n$-th symmetric (tensor) power of $P^\vee$,
        and let $SP^\vee=\bigoplus_{n=0}^\infty S^nP^\vee$ be
        the free symmetric algebra spanned by $P^\vee$.
        Then $\underline{P}:=\Spec SP^\vee$ fulfills all the requirements.
        If $\phi:Q\to P$ is a $K$-linear homomorphism, then
        the corresponding morphism $\underline{Q}\to\underline{P}$
        comes from the induced map $\bigoplus_{n\geq0}S^n\phi^\vee: SP^\vee\to SQ^\vee$.
        The  details are left as an exercise to the reader; see \cite[Ex.~20.2(2)]{InvBook}
        for the case where $K$ is a field.
    \end{proof}

    Suppose henceforth that $K$ is a Hausdorff topological commutative ring such that
    $\units{K}$ is open in $K$ and the \emph{inversion map} $k\mapsto k^{-1}:\units{K}\to \units{K}$
    is continuous.
    By Proposition~\ref{PR:realizing-projective-mods}, we can realize all
    finitely generated projective $K$-modules as the $K$-points of affine
    schemes over $K$, and topologize them using Proposition~\ref{PR:topologizing-schemes}.
    For every $P\in\rproj{K}$, denote by $\tau_P$ the  topology obtained in this way.
    It is immediate to check that this topology has the following properties:
    \begin{enumerate}
        \item[(1)] All $K$-linear homomorphisms are continuous.
        \item[(2)] For all $P,Q\in\rproj{K}$, the topology $\tau_{P\times Q}$ coincides
        with the product topology on $P\times Q$.
        \item[(3)] $\tau_K$ is  the topology of $K$ as a ring.
    \end{enumerate}
    As a result, if $P\in\rproj{K}$
    is isomorphic to a summand of $K^n$, then $\tau_P$ coincides with the subspace topology
    induced from the inclusion $P\to K^n$.

    \begin{prp}\label{PR:multilinear-is-continuous}
        Let $P_1,\dots,P_t,Q\in\rproj{K}$. Then any $K$-multilinear map
        $\mu:P_1\times\dots\times P_t\to Q$ is continuous.
    \end{prp}

    \begin{proof}
        It is enough to show that $\mu$ is induced by a $K$-morphism
        $\underline{\mu}:\underline{P_1}\times\dots\times\underline{P_t}\to\underline{Q}$.
        By Yoneda's Lemma, this holds if $\mu$ extends to a natural transformation
        from the points functor of $\underline{P_1}\times\dots\times\underline{P_t}$
        to the points functor of $\underline{Q}$.
        Namely, for all $L\in\Comm{K}$, there is a map $\mu_L:(P_1)_L\times\dots\times (P_t)_L\to Q_L$
        such that $\mu_K=\mu$ and $\{\mu_L\}_{L\in\Comm{K}}$ is a natural transformation
        from $[L\mapsto (\underline{P_1}\times\dots\times\underline{P_t})(L)=(P_1)_L\times\dots\times (P_t)_L]$
        to $[L\mapsto \underline{Q}(L)=Q_L]$.
        Indeed, it is easy  to check that the map $\mu_L(p_1\otimes \ell_1,\dots,p_t\otimes \ell_t)=\mu(p_1,\dots,p_t)\otimes \prod_i\ell_i$
        (with $p_1\in P_1$,\dots, $p_t\in P_t$, $\ell_1,\dots,\ell_t\in L$) fulfills these requirements.
    \end{proof}

    \begin{prp}\label{PR:top-K-mod}
        Every $P\in\rproj{K}$ is a topological $K$-module.
    \end{prp}

    \begin{proof}
        The addition and subtraction maps from $P\times P$ to $P$ are continuous
        because they are $K$-linear. The map $(p,k)\mapsto pk:P\times K\to P$
        is continuous because it is $K$-multilinear (Proposition~\ref{PR:multilinear-is-continuous}).
    \end{proof}

    \begin{prp}\label{PR:topological-algebra}
        Let $A$ be a $K$-algebra such that $A\in\rproj{K}$. Then:
        \begin{enumerate}
            \item[{\rm(i)}] $A$ is a topological $K$-algebra (when topologized as a f.g.\ projective $K$-module).
            \item[{\rm(ii)}] $\units{A}$ is open in $A$ and the map $a\mapsto a^{-1}:\units{A}\to\units{A}$
            is continuous.
        \end{enumerate}
    \end{prp}

    \begin{proof}
        (i)
        The algebra $A$ is a topological $K$-module by Proposition~\ref{PR:top-K-mod} and
        the multiplication in $A$ is continuous by
        Proposition~\ref{PR:multilinear-is-continuous}.

        (ii)
        If $A\cong \nMat{K}{n}$ as $K$-algebras, then this follows  from
        the fact that the determinant map is continuous, and the assumptions that
        $\units{K}$ is open in $K$ and the map $a\mapsto a^{-1}:\units{K}\to \units{K}$
        is continuous.

        For general $A$, choose $P\in\rproj{K}$ such that $A\oplus P$ is free and
        let $E=\End_K(A\oplus P)$ and $B=\End_K(A)$. Then $E\cong \nMat{K}{n}$ for some $n$,
        and hence $\units{E}$ is open in $E$ and the inversion map $i:\units{E}\to \units{E}$
        is continuous. Define $f:B\to E$ by $f(\phi)=(\phi\oplus 0_P)+(0_A\oplus\id_P)$
        and $g:E\to B$ by letting $g(\phi)$ be the unique element
        of $B$ satisfying $g(\phi)\oplus 0_P=(\id_A\oplus 0_P) \phi(\id_A\oplus 0_P)$.
        Then $f$ and $g$ are continuous. Since $\units{B}=f^{-1}(\units{E})$,
        $\units{B}$ is open in $B$, and since $g\circ i\circ f|_{\units{B}}$
        is the inversion map of $B$, the map $b\mapsto b^{-1}:\units{B}\to \units{B}$
        is continuous.

        Now define $t:A\to B=\End_K(A)$ by $t(a)=[x\mapsto ax]$
        and $s:B\to A$ by $s(b)=b(1_A)$. It is easy to check that $t$ is a homomorphism
        of $K$-algebras whose image is $\End_A(A_A)$, and $s\circ t=\id_A$.
        Thus, $A$ is a summand of $B$ via $t$,
        and hence $t$ is a closed embedding. Since $\End_A(A_A)\cap \units{B}=\units{\End_A(A_A)}$,
        we have
        $t^{-1}(\units{B})=\units{A}$. It follows that $\units{A}$ is open in $A$
        and $a\mapsto a^{-1}:\units{A}\to \units{A}$ is continuous.
    \end{proof}

    \begin{prp}\label{PR:point-wise-topology}
        Let $P\in\rproj{K}$. Then the topology of $\End_K(P)$
        coincides with the topology
        induced from the product topology on $P^P=\End_{\mathrm{Set}}(P)$.
    \end{prp}

    \begin{proof}
        Recall that a (finite) dual basis for a module $P$ consists of a finite collection $\{p_i,\psi_i\}_{i\in I}$
        such that $p_i\in P$, $\psi_i\in \dual{P}$ and $x=\sum_ip_i(\psi_ix)$ for all $x\in P$.
        A module is finitely generated projective if and only if it admits  a  dual basis
        (\cite[Lm.~2.9, Rm.~2.11]{La99}).

        Let
        $\{p_i,\psi_i\}_{i=1}^n$ be a dual basis for $P$ and write $E=\End_K(P)$. Define $\Phi:E\to P^n$
        and $\Psi:P^n\to E$ by
        $\Phi f= (fp_1,\dots,fp_n)$ and $(\Psi (x_1,\dots,x_n))x=\sum_i x_i(\psi_i x)$ (for all
        $x,x_1,\dots,x_n\in P$). It is straightforward to check that $\Psi\Phi=\id$, hence $E$ is a summand
        of $P^n$. Thus, $\tau_E$
        is induced from  $\tau_{P^n}$ via the embedding $\Phi$.

        Let $\tau$ denote the topology induced on $E$ from the embedding $E\to P^P$ (where $P^P$ is given the product topology).
        Since $\Phi$ factors through the embedding $E\to P^P$,
        and the factor map $P^P\to P^{n}$ is continuous, we have $\tau_{E}\subseteq \tau$.
        To see the converse, let $U$ be a $\tau$-neighborhood of some $f\in E$. Then there exists a neighborhood $V$
        of $f$ in $P^P$ whose inverse image in $E$ is $U$.
        By the definition of the product topology, there exist $p_{n+1},\dots,p_m\in P$
        and open sets $\{U_i\}_{i=n+1}^m\subseteq \tau_P$
        such that $fp_i\in U_i$ for all $n<i\leq m$,
        and $V\supseteq \prod_{i=n+1}^m U_{i}\times P^{P\setminus \{p_{n+1},\dots,p_m\}}$. Define
        $\psi_i=0\in \dual{P}$ for all $n<i\leq m$. Then
        $\{p_i,\psi_i\}_{i=1}^m$ is also a dual basis of $P$. Replacing
        $\{p_i,\psi_i\}_{i=1}^n$ (which was arbitrary) with $\{p_i,\psi_i\}_{i=1}^m$,
        we see that
        \[f\in\Phi^{-1}(\prod_{1\leq i\leq n}P\times \prod_{n<i\leq m} U_i)\subseteq U \ .\]
        It follows that every $\tau$-neighborhood of $f$ contains
        a $\tau_E$-neighborhood of $f$, so $\tau\subseteq\tau_E$.
    \end{proof}

    Suppose now that $F$ is  a subring of $K$.
    For any $Q\in\rproj{F}$, let $Q_K:=Q\otimes_F K\in \rproj{K}$. Similar notation
    will be applied to $F$-algebras and $F$-homomorphisms. We view $Q$ as  an $F$-submodule
    of $Q_K$ by identifying $x\in Q$ with $x\otimes 1_K$ (the map $x\mapsto x\otimes 1_K$ is injective
    for free modules, and hence for all projective modules).

    \begin{prp}\label{PR:top-scalar-ext}
        Let $Q\in\rproj{F}$.
        If $F$ is dense in $K$, then $Q$ is dense in $Q_K$.
    \end{prp}

    \begin{proof}
        Let $U$ be an open subset of $Q_K$ and let $x\in U$.
        Then $x=\sum_ix_ib_i$ for some $b_1,\dots,b_n\in K$  and $x_1,\dots,x_n\in Q$. Since $Q_K$ is a topological
        $K$-module, there are neighborhoods $b_i\in U_i\in\tau_K$ ($1\leq i\leq n$) such that $\sum_i x_iU_i\subseteq U$.
        Since $F$ is dense in $K$, there are $a_i \in U_i\cap F$ ($1\leq i\leq n$). Then $\sum_i x_ia_i\in U\cap Q$.
    \end{proof}

    \begin{prp}\label{PR:density-of-units}
        Assume that $\units{F}=F\cap \units{K}$ and
        let $A$ be an $F$-algebra such that $A\in\rproj{F}$. Then $\units{A}=A\cap\units{A_K}$.
        Furthermore,
        if  $F$ is dense in $K$, then $\units{A}$ is dense in $\units{A_K}$.
    \end{prp}

    \begin{proof}
        It is clear that $\units{A}\subseteq A\cap\units{A_K}$.
        To see the converse,
        let
        $a\in A\cap \units{A_K}$ and let $Q\in\rproj{F}$ be  such that $A\oplus Q\cong F^n$.
        Define $t:A\to\End_F(A)$ as in the proof of Proposition~\ref{PR:topological-algebra}.
        Then $a$ is invertible if and only if  $t(a)\oplus \id_Q\in\End_F(A\oplus Q)\cong\nMat{F}{n}$
        is invertible, which is equivalent to $\det(t(a)\oplus \id_Q)\in\units{F}$.
        Since $t(a)_K:A_K\to A_K$ is invertible, then so is $t(a)_K\oplus \id_{Q_K}\in \End_{K}(A_K\oplus Q_K)\cong \nMat{K}{n}$,
        hence $\det(t(a)_K\oplus \id_{Q_K})\in\units{K}$.
        On the other hand $\det(t(a)_K\oplus \id_{Q_K})=\det(t(a)\oplus \id_Q)\in F$,
        and so
        $\det(t(a)\oplus\id_Q)\in F\cap \units{K}=\units{F}$ (by assumption). Thus, $t(a)\oplus\id_Q$
        is invertible, as required.

        To finish, if $F$ is dense in $K$, then by Proposition~\ref{PR:top-scalar-ext}, $A$ is dense in $A_K$. By Proposition~\ref{PR:topological-algebra},
        $\units{A_K}$ is open $A_K$, so $\units{A}=A\cap\units{A_K}$ is dense in $\units{A_K}$.
    \end{proof}

    \begin{remark}\label{RM:rationally-closed}
        A subring $F\subseteq K$ satisfying $\units{F}=F\cap \units{K}$
        is called \emph{rationally closed} (in $K$). If $F$ is artinian then it automatically
        satisfies this condition, because being a non-unit in $F$ is equivalent to being a zero divisor.
    \end{remark}

    Let  $(A,\sigma,u,\Lambda)$ be a unitary $K$-algebra and let $(P,[f])$ be a unimodular quadratic
    space over $(A,\sigma,u,\Lambda)$.
    When the functor $L\mapsto O([f_L]):\Comm{K}\to\mathrm{Grp}$
    is representable,  there exists an affine group scheme over $K$, denoted $\uO([f])$,
    such that the functors $[L\mapsto O([f_L])]$ and $[L\mapsto \uO([f])(L)]$
    are isomorphic. In this case, we say that $\uO([f])$
    is \emph{defined as an affine $K$-scheme}.

    When $A\in\rproj{K}$ and $\Lambda$ is a summand of $A$, it is always
    true that $\uO([f])$ is defined as an affine $K$-scheme.
    Furthermore, in this case, $\End_A(P)$ is projective (being a $K$-summand of $\End_A(A^n,A^m)\cong A^{nm}$
    for some $n,m\in\N$)
    and the inclusion map $O([f])\to \End_A(P)$ is induced
    from a closed embedding $\uO([f])\to \underline{\End_{A}(P)}$.
    These somewhat technical facts are verified in the appendix.

\section{Weak Approximation}
\label{section:density}

    Let $K$ be a commutative semilocal topological ring and let $F$ be a dense subring
    of $K$. Let $(A,\sigma,u,\Lambda)$ be a unitary $F$-algebra and
    let $(P,[f])\in\Quad[u,\Lambda]{A,\sigma}$.
    In this section, we  prove a \emph{weak approximation theorem} for the group $O([f])$.
    Namely, we will show that under
    mild assumptions, the closure of the image of $O([f])$ in $O([f_K])$
    has finite index (which can be bounded effectively). This result will play an important role in the following sections.

    To give the flavor of the proof, let us sketch an ad-hoc proof in case $F$ and $K$ are fields of characteristic
    not $2$ and $(A,\sigma,u,\Lambda)=(F,\id_F,1,0)$: In this case, the Cartan--Dieudonn\'e Theorem implies that
    $O([f_K])$ is a generated by \emph{reflections}. Every reflection of $[f_K]$ can be topologically approximated
    by a reflection of $[f]$, and hence $O([f])$ is dense in $O([f_K])$.
    The proof that we give here follows essentially the same lines; the Cartan--Dieudonn\'e Theorem
    will be replaced with a certain generalization to semilocal base rings (see the introduction
    of \cite{Fi14A} for a brief survey of such results).

    We note that weak approximation theorems for reductive algebraic
    groups over general topological fields were studied by many authors, especially
    in the context of adjoint groups. See \cite{Thang96} and references therien for positive and
    negative results results.
    In fact, it is possible that our approximation result (Theorem~\ref{TH:density}, Corollary~\ref{CR:weak-approx-over-fields})
    can be deduced from such known results,
    at least when $F$ is a field and $A$ is a \emph{separable} $F$-algebra.
    The  methods we use here have the advantage of
    avoiding reductiveness issues, not using any valuation theory (any
    topological field works), and generalizing to semilocal rings.

\subsection{Generation by Pseudo-Reflections}
\label{subsection:reflections}

    Let $(A,\sigma,u,\Lambda)$ be a unitary ring and let $(P,[f])\in\Quad[u,\Lambda]{A,\sigma}$.
    For every $y\in P$ and $c\in \hat{f}(y)\cap\units{A}$, we define $s_{y,c}:P\to P$ by
    \[
    s_{y,c}(x)=x-y\cdot c^{-1}\cdot \tilde{h}_f(y,x)\qquad\forall \, x\in P\ .
    \]
    Following \cite[IV.\S2]{Bourbaki02LIEiv}, we call
    $s_{y,c}$  a \emph{pseudo-reflection} of $[f]$, or just a \emph{reflection} for short.
    It is well-known
    that $s_{y,c}\in O([f])$ and $s_{y,c}^{-1}=s_{y,c^\sigma u}$; see \cite[\S1]{Reiter75} or \cite[\S3]{Fi14A}.
    Denote by $O'([f])$ the subgroup of $O([f])$ generated by reflections.

    We now recall a theorem from \cite{Fi14A} describing the group $O'([f])$ in case $A$
    is {semilocal} and $P$ is free; see also \cite{Reiter75}
    for conditions guaranteeing that $O'([f])=O([f])$.
    Recall
    that $A$ is \emph{semilocal} if $A/\Jac(A)$ is semisimple artinian ($\Jac(A)$ denotes the Jacobson radical). If
    in addition idempotents lift modulo $\Jac(A)$, then $A$ is called \emph{semiperfect}.
    For example,
    all one-sided artinian rings are semiperfect; see \cite[\S2.7]{Ro88} for further examples and details.

\medskip

    Assume $A$ is semilocal.
    We set some general notation: Let $\quo{A}:=A/\Jac(A)$, let $\quo{\sigma}$ be the involution
    induced by $\sigma$ on $\quo{A}$, and set $\quo{a}=a+\Jac(A)$ for all $a\in A$.
    Then $(\quo{A},\quo{\sigma},\quo{u},\quo{\Lambda})$ is a semisimple unitary ring,
    hence by Proposition~\ref{PR:factorization-of-unitary-rings}, it factors into a product
    \[
    (\quo{A},\quo{\sigma},\quo{u},\quo{\Lambda})=\prod_{i=1}^t(A_i,\sigma_i,u_i,\Lambda_i)
    \]
    where the factors are simple artinian unitary rings (see~\ref{subsectnio:csa-s}).
    We further write $A_i=\nMat{D_i}{n_i}$ where $D_i$ is a division ring or a product of
    a division ring and its opposite.

    Every quadratic space $(P,[f])$ over $(A,\sigma,u,\Lambda)$ gives rise
    to a quadratic space $(\quo{P},[\quo{f}])$ over $(\quo{A},\quo{\sigma},\quo{u},\quo{\Lambda})$.
    Namely, $\quo{P}=P/P\Jac(A)$ and $\tilde{\quo{f}}$ is given by $\tilde{\quo{f}}(\quo{x},\quo{y})=\quo{\tilde{f}(x,y)}$
    for all $x,y\in P$
    (where  $\quo{x}=x+P\Jac(A)$).
    This in turn gives rise to quadratic spaces $(P_i,[f_i])_{i=1}^t$ over $(A_i,\sigma_i,u_i,\Lambda_i)_{i=1}^t$;
    if one writes $\quo{P}=\prod_i P_i$ with $P_i\in\rproj{A_i}$, then $\tilde{f}_i$ is just the restriction
    of $\tilde{\quo{f}}$ to $P_i\times P_i$. It is well-known that if $(P,[f])$ is unimodular,
    then so are $(P,[f_i])_{i=1}^t$ (see \cite[Lm.~4.3(ii)]{Fi14A} and more generally
    \cite[\S{}I.7.1]{Kn91}).

    Every isometry $\phi\in O([f])$ induces an isometry $\quo{\phi}\in O([\quo{f}])$
    given by $\quo{\phi}(\quo{x})=\quo{\phi x}$. Restricting $\quo{\phi}$ to $P_i$ yields
    an isometry $\phi_i\in O([f_i])$. It is easy  to check that $\phi\mapsto \phi_i:O([f])\to O([f_i])$
    is a group homomorphism.
    When $(A_i,\sigma_i,u_i,\Lambda_i)$ is split-orthogonal (see~\ref{subsectnio:csa-s}) and
    $(P,[f])$ is unimodular, we define
    \[
    \begin{array}{rccl}
        \Delta_i=\Delta_{i,[f]}: & O([f]) & \to & \Z/2\Z \\
        & \phi & \mapsto & \Delta_{[f_i]}(\phi_i)
    \end{array}
    \]
    where $\Delta_{[f_i]}$ is the Dickson map of $[f_i]$ (see~\ref{subsection:Dickson-inv}).
    If $\calI$ is any subset of $\{1,\dots,t\}$ consisting of indices
    $i$ for which $(A_i,\sigma_i,u_i,\Lambda_i)$ is split-orthogonal, we define
    \[
    \Delta_{\calI}=\Delta_{\calI,[f]}:O([f])\to (\Z/2\Z)^{\calI}
    \]
    by $\Delta_{\calI}(\phi)=(\Delta_i(\phi))_{i\in\calI}$.

    \begin{thm}
        \label{TH:gen-by-reflections}
        Let $(P,[f])\in\Quad[u,\Lambda]{\sigma,A}$.  Assume that $P$ is free
        and for all $1\leq i\leq t$, $D_i$ is not isomorphic to $\F_2$ or $\F_2\times\F_2$.
        Let $\calI=\calI(A)$ denote the set of $i$-s for which $(A_i,\sigma_i,u_i,\Lambda_i)$
        is split-orthogonal, and let $\xi=(n_i+2\Z)_{i\in\calI}=(\deg A_i+2\Z)_{i\in\calI}\in (\Z/2\Z)^{\calI}$.
        Then
        \[
        O'([f])=\Delta_{\calI}^{-1}(\{0,\xi\})\ .
        \]
        When $A$ is semiperfect, the theorem holds under the milder assumption that
        $P_i\neq 0$ for all $1\leq i\leq t$ ($P$ does not have to be free).
        Furthermore, in this case, $\Delta_{\calI}$ is onto.
    \end{thm}

    \begin{proof}
        See \cite[Th.~5.8, Th.~5.10]{Fi14A}.
    \end{proof}

    \begin{cor}[{\cite[Cr.~5.9]{Fi14A}}]
        \label{CR:gen-by-reflections}
        In the setting of Theorem~\ref{TH:gen-by-reflections}, $[O([f]):O'([f])]$ is
        a finite power of $2$.
    \end{cor}

\subsection{The Dickson Map and Transfer}
\label{subsection:dickson-and-transfer}

    Keep the setting of~\ref{subsection:reflections}.
    We now verify that the map $\Delta_{\calI}$ of Theorem~\ref{TH:gen-by-reflections} is compatible
    with transfer in the sense of~\ref{subsection:transfer}.

\medskip

    Let $(P,[f])\in\Quad[u,\Lambda]{A,\sigma}$
    and define $(B,\tau,1,\Gamma)$ as in~\ref{subsection:transfer}
    with $(Q,h)=(P,h_f)$. Write also $[g]=[\mathrm{T}_hf]$ and identify
    $O([f])$ with $O([g])$ as in Proposition~\ref{PR:transfer}.
    Similarly to~\ref{subsection:reflections}, we write
    $(\quo{B},\quo{\tau},\quo{1},\quo{\Gamma})=\prod_{j=1}^s(B_j,\tau_j,1,\Gamma_j)$
    where the factors are simple artinian unitary rings.
    This gives rise to quadratic spaces $(B_i,[g_i])\in\Quad[1,\Gamma_i]{B_i,\tau_i}$.
    Finally, let $\calI(B)$ be the set of $j$-s for which
    $(B_j,\tau_j,1,\Gamma_j)$ is split-orthogonal and
    let $\calI(P)$ be the set of $i$-s for which
    $(A_i,\sigma_i,u_i,\Lambda_i)$ is split orthogonal and $P_i\neq 0$.

    \begin{prp}\label{PR:compatiability-with-Dickson-inv}
        There is an isomorphism $\calI(B)\cong \calI(P)$ such that the induced
        isomorphism $(\Z/2\Z)^{\calI(B)}\cong (\Z/2\Z)^{\calI(P)}$
        makes the following diagram commute:
        \[
        \xymatrixcolsep{5pc}\xymatrix{
        O([f]) \ar@{=}[d] \ar[r]^{\Delta_{\calI(P),[f]}} & (\Z/2\Z)^{\calI(P)} \ar[d] \\
        O([g]) \ar[r]^{\Delta_{\calI(B),[g]}} & (\Z/2\Z)^{\calI(B)}
        }
        \]
    \end{prp}

    \begin{proof}
        We first claim that $\Jac(B)=\Hom_A(P,P\Jac(A))$
        and $B/\Jac(B)\cong\End_{\quo{A}}(\quo{P})$. This is a standard argument:
        It is easy to see that
        for all $P,Q\in\rproj{A}$, we have a natural isomorphism
        $\Hom_A(P,Q)/\Hom_A(P,Q\Jac(A))\cong\Hom_{\quo{A}}(\quo{P},\quo{Q})$
        (check this when $P=Q=A_A$ and then extend to general $P$ and $Q$ using the naturality).
        Thus, $B/\Hom_A(P,P\Jac(A))\cong\End_{\quo{A}}(\quo{P})$ (as rings). Since $\quo{A}$
        is semisimple, $B/\Hom_A(P,P\Jac(A))$ is semisimple and hence $\Jac(B)\subseteq\Hom_A(P,P\Jac(A))$.
        To see the other inclusion, observe that if $\phi\in 1+\Hom_A(P,P\Jac(A))$, then $\im(\phi)+P\Jac(A)=P$,
        so by Nakayama's Lemma, $\phi$ is onto. Since $P$ is projective, $\phi$ admits a right inverse. It
        follows that $1+\Hom_A(P,P\Jac(A))$ consists of right-invertible elements, hence $\Hom_A(P,P\Jac(A))\subseteq \Jac(B)$.

        Since $\quo{A}=\prod_iA_i$,
        we have $\quo{B}\cong\End_{\quo{A}}(\quo{P})=\prod_{i=1}^t\End_{A_i}(P_i)$ (it is possible that $P_i=0$).
        By~\ref{subsection:transfer}, the hermitian space $(P_i,h_{f_i})$ induces a
        unitary ring structure on $\End_{A_i}(P_i)$. The resulting unitary ring structure
        on $\prod_{i=1}^t\End_{A_i}(P_i)$ is easily seen to coincide with the one on $\quo{B}$.
        We may therefore identify $\quo{B}$ and $\prod_{i=1}^t\End_{A_i}(P_i)$ as unitary
        rings. Let $\calJ=\{1\leq i\leq t\where P_i\neq 0\}$. By Proposition~\ref{PR:transfer-preserve-split-orth},
        the rings $\End_{A_i}(P_i)$ are simple artinian as unitary rings, and $\End_{A_i}(P_i)$ is split-orthogonal if
        and only $A_i$ is split-orthogonal and $P_i\neq 0$. This gives rise to a bijection $\alpha:\{1,\dots,s\}\to \calJ$
        such that for all $j$, $B_j\cong\End_{A_{\alpha(j)}}(P_{\alpha(j)})$ as unitary rings,
        and $B_j$ is split-orthogonal if and only if $A_{\alpha(j)}$ is split-orthogonal.
        In particular, $\alpha$ restricts to a bijection between $\calI(B)$ and $\calI(P)$.

        The commutativity of the diagram follows directly from the definition of the maps
        $\Delta_{\alpha(j),[f]}$ and $\Delta_{j,[g]}$ ($j\in\calI(B)$) (see~\ref{subsection:Dickson-inv}); it depends only on the isomorphism class
        of the ring $\End_{A_{\alpha(j)}}(P_{\alpha(j)})\cong B_j\cong \End_{B_j}(B_j)$.
    \end{proof}

\subsection{Weak Approximation}
\label{subsection:density}

    We now use  Theorem~\ref{TH:gen-by-reflections} to prove a weak approximation theorem.

    \begin{thm}\label{TH:density}
        Let $K$ be a semilocal commutative Hausdorff topological ring, let $F$ be a subring of $K$,
        let $(A,\sigma,u,\Lambda)$ be a unitary $F$-algebra and let $(P,[f])\in\Quad[u,\Lambda]{A,\sigma}$.
        Assume that:
        \begin{enumerate}
            \item[$(0)$] $\F_2$ is not an epimorphic image of $K$,
            \item[$(1)$] $\units{K}$ is open in $K$ and the map $a\mapsto a^{-1}:\units{K}\to\units{K}$
            is continuous,
            \item[$(2)$] $F$ is dense in $K$ and $F\cap \units{K}=\units{F}$,
            \item[$(3)$] $A$ is a finitely generated projective $F$-module and $\Lambda$ is an $F$-summand of $A$.
        \end{enumerate}
        Topologize $\End_{A_K}(P_K)$ as in section~\ref{section:topology}
        and give $O([f_K])$ the subspace topology. In addition, define
        $\calI(P_K)$ as in \ref{subsection:dickson-and-transfer} (with $A_K$ in the role of $A$). Then,
        viewing $O([f])$ as a subgroup of $O([f_K])$ via $\phi\mapsto \phi_K$,
        we have
        \[
        \overline{O([f])}\supseteq \ker\Delta_{\calI(P_K),[f_K]}\ .
        \]
        In particular, $[O([f_K]):\overline{O([f])}]$ is a finite power of $2$.
    \end{thm}

    \begin{proof}
        Observe first that condition (3) implies that $A_K$ and $\Lambda_K$ are projective $K$-modules.
        Since $P_K\in\rproj{A_K}$, $\End_{A_K}(P_K)$ is a projective $K$-module
        and hence can be topologized as in section~\ref{section:topology}.
        Furthermore, $\units{\End_{A_K}(P_K)}$ is a topological group by Proposition~\ref{PR:topological-algebra},
        so $O([f_K])$ is a topological group as well. Alternatively, we can topologize $O([f_K])$ directly
        using Proposition~\ref{PR:topologizing-schemes} and the fact that $\uO([f_K])$ is defined as an affine $K$-scheme by condition (3).
        This gives the same topology.

        Also note that $A_K$ is indeed semilocal:
        By \cite[Cr.~II.4.2.4]{Kn91},
        $A_K\Jac(K)\subseteq \Jac(A_K)$. Since $A_K/A_K\Jac(K)$ is a finitely generated $K/\Jac(K)$-module,
        it is an artinian ring. It follows that $A_K/\Jac(A_K)$ is an epimorphic image of an artinian ring (namely,
        $A_K/A_K\Jac(K)$),
        and hence semisimple artinian.

\smallskip

        We now turn to the proof itself:
        By Proposition~\ref{PR:compatiability-with-Dickson-inv}, we may apply
        transfer (see~\ref{subsection:transfer}) and hence assume that $P$ is free.
        Now, by Theorem~\ref{TH:gen-by-reflections} and condition (0), every isometry
        in $\ker\Delta_{\calI(P_K)}$ is a product of reflections, so it is enough to show that
        every neighborhood of a reflection of $[f_K]$ contains a reflection
        of $[f]$.
        Indeed, let $s=s_{y,c}$ be a reflection of $[f_K]$. By Proposition~\ref{PR:point-wise-topology},
        every neighborhood of $s$ contains a set of the form
        $\{\psi\in O([f_K])\suchthat \psi x_i-s x_i\in U_i~\forall 1\leq i\leq n\}$,
        where $x_1,\dots,x_n\in P_K$ and $U_1,\dots,U_n$ are neighborhoods of $0$ in $P_K$.
        Write $c=\tilde{f}_K(y,y)+\gamma$ for $\gamma\in \Lambda_K$.
        Then
        \[sx_i=x_i-y\cdot (\tilde{f}_K(y,y)+\gamma)^{-1}\cdot \tilde{h}_{f_K}(y,x_i)\ .\]
        By Proposition~\ref{PR:multilinear-is-continuous}, $\tilde{f}_K$ and $\tilde{h}_{f_K}$
        are continuous, and hence, by Propositions~\ref{PR:top-K-mod} and~\ref{PR:topological-algebra}, for each $i$, the function
        \begin{equation}\label{EQ:reflection-functions}
        (y,\gamma)\mapsto x_i-y\cdot (\tilde{f}_K(y,y)+\gamma)^{-1}\cdot \tilde{h}_{f_K}(y,x_i)
        \end{equation}
        is continuous wherever defined. Furthermore,  its domain is an open subset of $P_K\times \Lambda_K$ by
        Proposition~\ref{PR:topological-algebra}(ii) (since $\units{A_K}$ is open in $A_K$).
        By Proposition~\ref{PR:top-scalar-ext}, $P\times \Lambda$ is dense in $P_K\times \Lambda_K$.
        Thus, there is $z\in P$ and $\lambda\in\Lambda$ such that $\tilde{f}(z,z)+\lambda\in\units{A_K}$ and
        \[
        \left[
        x_i-y\cdot (\tilde{f}_K(y,y)+\gamma)^{-1}\cdot \tilde{h}_{f_K}(y,x_i)
        \right]-
        \left[
        x_i-z\cdot (\tilde{f}(z,z)+\lambda)^{-1}\cdot \tilde{h}_f(z,x_i)
        \right]\in U_i
        \]
        for all $i$. Write $d=\tilde{f}(z,z)+\lambda$. Then $d\in A\cap \units{A_K}=\units{A}$ (Proposition~\ref{PR:density-of-units}, condition (2)),
        and hence $s_{z,d}$ is a reflection of $[f]$. Since
        $sx_i-s_{z,d}x_i\in U_i$
        for all $i$, we are done.
    \end{proof}

    \begin{remark}
        In the setting of Theorem~\ref{TH:density}:
        \begin{enumerate}
        \item[(i)] The condition $F\cap \units{K}=\units{F}$  always holds when
        $F$ is a field or, more generally, an artinian ring (cf.~Remark~\ref{RM:rationally-closed}).
        \item[(ii)] The condition that $\Lambda$ is an $F$-summand of $A$ holds when $2\in\units{A}$.
        Indeed, in this case,
        it is easy to check that $\Lambda=\Lambda^{\min}(u)$ and $A=\Lambda^{\min}(u)\oplus \Lambda^{\min}(-u)$.
        \end{enumerate}
    \end{remark}

    \begin{remark}\label{RM:affine-grp-scheme}
        Condition (3) of Theorem~\ref{TH:density}
        implies that
        $\uO([f])$ is defined an affine $F$-scheme (see the end of Section~\ref{section:topology}).
        At this level or generality, we do not know
        whether $\Delta_{\calI(P_K),[f_K]}:O([f_K])\to (\Z/2\Z)^{\calI(P_K)}$
        is induced from a morphism  of affine group schemes  over $K$. However, this is true when $K$
        is a field:
        Let $A_i$ be a split-orthogonal factor of the semisimple unitary ring $A_K/\Jac(A_K)$.
        The map $O([f_K])\to O([(f_K)_i])$ is a restriction of the
        standard map $\End_{A_K}(P_K)\to\End_{A_i}(P_i)$, which is $K$-linear.
        Thus, $O([f_K])\to O([(f_K)_i])$ is induced by a morphism of affine groups schemes over $K$
        (cf.\ Proposition~\ref{PR:realizing-projective-mods}).
        Since $\Delta_{[(f_K)_i]}:O([(f_K)_i])\to \Z/2\Z$ is a morphism of algebraic groups
        over $K$ (see~\ref{subsection:Dickson-inv}), it follows that
        $\Delta_{[f_K],i}:O([f_K])\to \Z/2\Z$ is induced by a morphism of affine groups schemes over $K$.
        Letting $i$ range over $\calI(P_K)$, we see that this also holds for $\Delta_{\calI(P_K)}:O([f_K])\to(\Z/2\Z)^{\calI(P_K)}$,
        where $\Z/2\Z$ is viewed as a constant group scheme over $K$.

        This argument also works when $K$ is a product of  fields $K_1\times\dots \times K_r$.
        However, in this case, $(\Z/2\Z)^{\calI(P_K)}$ should be realized as $\prod_{j=1}^r(\Z/2\Z)^{\calI(P_{K_j})}$
        where $(\Z/2\Z)^{\calI(P_{K_j})}$ is a constant group scheme over $K_j$.
    \end{remark}

    The following corollary shows that weak approximation holds for the \emph{connected component}
    of $\uO([f])$ when $F$ is a field (see \cite[\S6.7]{Water79IntroAffineGrpSchemes} for the definition).

    \begin{cor}\label{CR:weak-approx-over-fields}
        Keep the setting of Theorem~\ref{TH:density}, and assume further that $F$ is a field
        and $K$ is a product of finitely many fields.
        Let $\uO^+:=\uO([f])^+$ be the connected component of $\uO:=\uO([f])$.
        Give
        $\uO^+(K)$ the topology of Proposition~\ref{PR:topologizing-schemes}.
        Then
        $
        \uO^+(F)$
        is dense in $\uO^+(K)
        $.
    \end{cor}

    \begin{proof}
        We will use of the algebraic group $\pi_0(\uO)$; see \cite[\S6.5]{Water79IntroAffineGrpSchemes} for its definition. The exact sequence
        of algebraic groups $1\to \uO^+\to \uO\to \pi_0(\uO)\to 1$ gives rise to an exact
        sequence $1\to \uO^+(K)\to \uO(K)\to \pi_0(\uO)(K)$ and all
        the morphisms are continuous by Proposition~\ref{PR:topologizing-schemes}.
        Since $\pi_0(\uO)(K)$ is finite and $K$ is Hausdorff, $\pi_0(\uO)(K)$ is discrete, and hence $\uO^+(K)$
        is open and closed in $\uO(K)$.

        By Remark~\ref{RM:affine-grp-scheme}, the map $\Delta_{\calI(P_K)}$ is induced
        from a homomorphism of algebraic groups $\underline{\Delta}:\uO\to(\Z/2\Z)^{\calI(P_K)}$.
        Furthermore, $(\Z/2\Z)^{\calI(P_K)}$ is \'{e}tale over $\Spec K$, and hence  $\underline{\Delta}$
        factors through $\uO\to\pi_0(\uO)$ \cite[Th.~6.7]{Water79IntroAffineGrpSchemes}.
        It follows that $\uO^+(K)\subseteq \ker\Delta_{\calI(P_K)}$.
        By Theorem~\ref{TH:density}, this implies $\uO^+(K)\subseteq \overline{\uO(F)}$,
        and since $\uO^+(K)$ is open and closed in $\uO(K)$,
        we get $\uO^+(K)=\overline{\uO(F)\cap \uO^+(K)}=\overline{\uO^+(F)}$.
    \end{proof}

\section{A Double Coset Argument}
\label{section:double-cosets}

    Given a commutative ring  $R$, a family of commutative $R$-algebras
    $\calL$, and a  quadratic space $(P,[f])$ over a unitary $R$-algebra,
    we let the \emph{$\calL$-genus} of $(P,[f])$,
    denoted $\gen_{\calL}(P,[f])$,
    be the collection of isomorphism
    classes of  quadratic spaces $(P',[f'])$ that become isometric to
    $(P,[f])$ after applying $\scalarExt{L}{R}$ for every $L\in\calL$. This generalizes
    the  genus considered in the introduction.

    As preparation for the patching theorem of the next section, this section relates the $\calL$-genus of
    $(P,[f])$ to double cosets in a certain group, in case $\calL$ consists of two
    algebras satisfying certain assumptions.
    The argument  resembles some descent methods used in number theory
    and algebraic geometry, which generalize much beyond the scope of quadratic spaces.
    We comment about this  in detail in Remark~\ref{RM:double-cosets}.

\medskip

    Recall that a commutative square of abelian groups
    \begin{equation}\label{EQ:set-cart-square}
    \xymatrix{
    B \ar[r]|\phi & D \\
    A \ar[u]|\beta \ar[r]|\gamma & C \ar[u]|\psi
    }
    \end{equation}
    is \emph{cartesian} if $A$ is the pullback of $\phi$ and $\psi$. Namely, for all
    $b\in B$, $c\in C$ with $\phi b=\psi c$, there exists unique $a\in A$ with $\beta a=b$ and $\gamma a=c$.
    We  say that the square is \emph{onto} if $\phi(B)+\psi(C)=D$.

    The
    properties just defined can  be explained via exactness of the following sequence
    \begin{equation*}
    0\to A\xrightarrow{a\mapsto (\beta a\oplus\gamma a)} B\oplus C\xrightarrow{(b\oplus c) \mapsto (\phi b-\psi c)} D\to 0\ .
    \end{equation*}
    Namely, the square \eqref{EQ:set-cart-square}
    is cartesian if the sequence is exact on the left and on the middle,
    and onto if it is exact on the right.

\medskip

    Throughout, we fix a cartesian and onto square of commutative rings
    \begin{equation}\label{EQ:rings-cart-square}
    \xymatrix{
    S \ar[r] & K \\
    R \ar[u] \ar[r] & F \ar[u]
    }
    \end{equation}
    In addition, $(A,\sigma,u,\Lambda)$ is  a unitary $R$-algebra such that $A$ is flat as an $R$-module.

    We shall use the notation of~\ref{subsection:unitary-algs} for scalar extension of modules, homomorphisms, quadratic forms, etcetera.
    Furthermore, for brevity,  for all $P\in\rproj{A}$, we set
    \[
    \calS(P):=\Hom(P,P^*)\qquad\text{and}\qquad \calQ(P)=\Hom(P,P^*)/\Lambda_P\ .
    \]
    Similar notation will be used for modules over $A_S$, $A_F$ and $A_K$.
    Recall from~\ref{subsection:unitary-algs} that we have  scalar extension maps
    \[f\mapsto f_S:\calS(P)\to \calS(P_S),\qquad [f]\mapsto [f_S]:\calQ(P)\to \calQ(P_S)\]
    and likewise for any pair of the rings $R,S,F,K$ connected by a homomorphism.

    \begin{lem}\label{LM:cart-square-I}
        For any flat $R$-module $M$, the following square is cartesian and onto
        \begin{equation}\label{EQ:module-cart-square}
        \xymatrix{
        M_S \ar[r] & M_K \\
        M \ar[u] \ar[r] & M_F \ar[u]
        }
        \end{equation}
    \end{lem}

    \begin{proof}
        As explained above, the square \eqref{EQ:rings-cart-square} gives rise to an exact sequence
        \begin{equation}\label{EQ:cartesion-exact-sequence}
        0\to R\xrightarrow{~} S\oplus K\xrightarrow{~} F\to 0
        \end{equation}
        The lemma follows by tensoring with $M$, which preserves exactness since $M$ is flat.
    \end{proof}

    \begin{lem}\label{LM:cart-square-II}
        Let $P,P'\in\rproj{A}$. Consider the following squares induced by \eqref{EQ:rings-cart-square}:
        \[
        \xymatrixcolsep{1pc}\xymatrix{
        \Hom_{A_S}(P_S,P'_S) \ar[r] & \Hom_{A_K}(P_K,P'_K) \\
        \Hom_{A}(P,P') \ar[u] \ar[r] & \Hom_{A_F}(P_F,P'_F) \ar[u]
        }
        \quad
        \xymatrix{
        \calS(P_S) \ar[r] & \calS(P_K) \\
        \calS(P) \ar[u] \ar[r] & \calS(P_F) \ar[u]
        }
        \quad
        \xymatrix{
        \calQ(P_S) \ar[r] & \calQ(P_K) \\
        \calQ(P) \ar[u] \ar[r] & \calQ(P_F) \ar[u]
        }
        \]
        Then:
        \begin{enumerate}
            \item[{\rm(i)}] The left and middle squares are  cartesian and onto.
            Furthermore, if $\psi\in\Hom_{A}(P,P')$ is such that $\psi_F$ and $\psi_S$
            are invertible, then $\psi$ is invertible.
            \item[{\rm(ii)}] Provided $K$ is flat as an $R$-module, the right square is cartesian and onto.
        \end{enumerate}
    \end{lem}

    \begin{proof}
        (i) By Lemma~\ref{LM:natural-isomorphism}, the middle square is  a special case
        of the left square. This lemma also allows us to identify $\Hom_{A_T}(P_T,P'_T)$ with $\Hom_A(P,P')_T$
        for any commutative $R$-algebra $T$. Now, by Lemma~\ref{LM:cart-square-I}, in order to prove that
        the left square is cartesian and onto,
        it is enough to show that $\Hom_A(P,P')$ is a flat $R$-module.
        Indeed, $\Hom_A(P,P')$ is a summand of $\Hom_A(A^n,A^m)\cong A^{nm}$ (as $R$-modules)
        for some $n,m\in \N$, and $A$ is flat by assumption.

        Next, assume $\psi\in\Hom_{A}(P,P')$ is such that $\psi_F$ and $\psi_S$
        are invertible. Then in $\Hom_{A_K}(P'_K,P_K)$ we have $(\psi_S^{-1})_K=(\psi_F^{-1})_K$,
        hence there exists $\phi\in\Hom_A(P',P)$ with $\phi_S=\psi_S^{-1}$ and $\phi_F=\psi_F^{-1}$.
        We clearly have $(\phi\psi-\id_P)_S=0$ and $(\phi\psi-\id_P)_F=0$, so by cartesianity of the left
        square (in case $P=P'$), we have $\phi\psi=\id_P$. Likewise, $\psi\phi=\id_{P'}$, so $\psi$ is invertible.

        (ii) By
        Proposition~\ref{PR:scalar-ext-for-quad-forms},
        we may identify $\calS(P_T)$ with $\calS(P)_T$ and $\calQ(P_T)$ with $\calQ(P)_T$ for any $R$-algebra $T$.
        Consider the following diagram, which is obtained
        by tensoring the top row with \eqref{EQ:cartesion-exact-sequence}:
        \[
        \xymatrix{
        0 \ar[r] & \Lambda_P \ar[r]\ar[d] & \calS(P) \ar[r]\ar[d]^{\eta} & \calQ(P) \ar[r]\ar[d]^\psi & 0 \\
        &  (\Lambda_P)_S\oplus (\Lambda_P)_F \ar[r]\ar[d] & \calS(P)_S\oplus \calS(P)_F \ar[r]\ar[d] & \calQ(P)_S\oplus \calQ(P)_F \ar[r]\ar[d] & 0 \\
        0 \ar[r] & (\Lambda_P)_K \ar[r] & \calS(P)_K \ar[r] &  \calQ(P)_K \ar[r] & 0
        }
        \]
        The first two rows are clearly exact, and the third row is exact since $K$ is flat. In addition,
        all columns are exact in the middle and  on the bottom (once adding a zero object), and by (i),
        $\eta$ is injective. We only need to prove that $\psi$ is injective, and this follows by standard diagram chasing.
        (Specifically, assume $a\in\calQ(P)$ satisfies $\psi a=0$. Let $b\in\calS(P)$ be an inverse image of $a$
        and let $c=\eta b$. Then the image of $c$ in $\calQ(P)_S\oplus \calQ(P)_F$ is $\psi a=0$,
        hence $c$ has an inverse image $d\in(\Lambda_P)_S\oplus (\Lambda_P)_F$. The image of $d$ in $\calS(P)_K$
        is the image of $c=\eta b$ in $\calS(P)_K$, which is $0$. Thus, the image of $d$ in $(\Lambda_P)_K$ is $0$
        (since the third row is exact on the left). Let $e\in \Lambda_P$ be an inverse image of $d$, and let
        $f$ be the image of $e$ in $\calS(P)$. Then $\eta f=c=\eta b$, so $f=b$ (since $\eta$ is injective).
        This means that $a$ is the image of $e\in\Lambda_P$, and hence equals $0$.)
    \end{proof}

    We remark that part (i) of the lemma implies the following corollary.

    \begin{cor}\label{CR:unimodularity-in-genus}
        Let $(P,[f])$ be a quadratic space over $(A,\sigma,u,\Lambda)$
        and let $(P',[f'])\in\gen_{S,F}(P,[f])$. If $(P,[f])$ is unimodular, then so is
        $(P',[f'])$.
    \end{cor}

    \begin{proof}
        Write $h=h_f$ and $h'=h_{f'}$. Then both $h_S$ and $h_F$ are invertible,
        hence $h'_S$ and $h'_F$ are invertible (because $(P',[f'])\in\gen_{S,F}(P,[f])$).
        Thus, by Lemma~\ref{LM:cart-square-II}(i), $h'$ is invertible.
    \end{proof}

    \begin{notation}\label{NT:patching-groups}
        Let $(P,[f])$ be a quadratic space over $(A,\sigma,u,\Lambda)$. We set:
        \begin{align*}
        O_S&= \{\phi_K\where \phi\in O([f_S])\}, & G_S&=\{\phi_K\where \phi\in\units{\End_{A_S}(P_S)}\}, \\
        O_F&= \{\phi_K\where \phi\in O([f_F])\}, & G_F&=\{\phi_K\where \phi\in\units{\End_{A_F}(P_F)}\}, \\
        O_K&= O(P_K,[f_K]), & G_K&=\units{\End_{A_K}(P_K)}.
        \end{align*}
    \end{notation}

    \begin{thm}\label{TH:patching-general}
        Let $(P,[f])$ be a quadratic space over $(A,\sigma,u,\Lambda)$,
        and
        assume that $K$ is a flat $R$-module.
        Then, in the notation of~\ref{NT:patching-groups},
        there exists an injection
        \[
        \Phi:\gen_{S,F}(P,[f])\to O_S\!\setminus\! O_K/O_F,\
        \]
        constructed in the proof.
        When $O_K\subseteq G_SG_F$, the map $\Phi$ is bijective and
        every $(P',[f'])\in \gen_{S,F}(P,[f])$ satisfies $P'\cong P$.
    \end{thm}

    \begin{proof}
        We shall use the following special notation: If $(Q,[g])$ is a quadratic
        space and $\phi\in\Hom(Q',Q)$, define
        \[
        [g]\bullet \phi=[\phi^* g\phi]\ .
        \]
        We clearly have $[g]\bullet(\phi\psi)=
        ([g]\bullet \phi)\bullet\psi$ when both sides are well-defined.

\medskip

        We first construct $\Phi$. Let $(P',[f'])$ be a representative for an isomorphism
        class in $\gen_{S,F}(P,[f])$. Then there are
        isometries $\phi:(P'_S,[f'_S])\to (P_S,[f_S])$ and $\psi:(P'_F,[f'_F])\to (P_F,[f_F])$, and
        we have
        $\phi_K\psi_K^{-1}\in O_K$.
        Using this, define
        \[
        \Phi(P',[f'])=O_S\phi_K\psi_K^{-1}O_F\ .
        \]
        Observe that $\Phi(P',[f'])$ is independent of the choices of $\phi$ and $\psi$.
        Indeed, if $\theta:(P'_S,[f'_S])\to (P_S,[f_S])$ and $\xi:(P'_F,[f'_F])\to(P_F,[f_F])$ are
        other isometries, then $\phi_K\theta^{-1}_K \in O_S$ and $\psi_K\xi^{-1}_K\in O_F$, hence we get
        \[
        O_S\theta_K\xi_K^{-1}O_F=O_S(\phi_K\theta^{-1}_K)\theta_K\xi_K^{-1}(\psi_K\xi^{-1}_K)^{-1}O_F=
        O_S\phi_K\psi_K^{-1}O_F\ .
        \]
        To see that $\Phi$ is well-defined up to isometry (over $R$), let
        $\eta:(P'',[f''])\to (P',[f'])$ be an isometry. Then
        there are isometries
        $\phi\eta_S: (P''_S,[f''_S])\to (P_S,[f_S])$ and $\psi\eta_F:(P''_F,[f''_F])\to(P_F,[f_F])$. We now have
        \[
        \Phi(P'',[f''])=O_S(\phi\eta_S)_K(\psi \eta_F)^{-1}_KO_F=O_S\phi_K\psi_K^{-1}O_F=\Phi(P',[f']),
        \]
        as required.

        Next, we verify that $\Phi$ is injective. Assume that $\Phi(P',[f'])=\Phi(P'',[f''])$,
        let $\phi,\psi$ be  as above,  and let $\phi':(P''_S,[f''_S])\to (P_S,[f_S])$,
        $\psi':(P''_F,[f''_F])\to (P_F,[f_F])$. Then
        \[
        O_S\phi_K\psi^{-1}_KO_K=O_S\phi'_K\psi'^{-1}_KO_F,
        \]
        so we can write $\phi_K\psi^{-1}_K=\alpha_K\phi'_K\psi'^{-1}_K\beta_K^{-1}$ for
        $\alpha\in O(P_S,[f_S])$ and $\beta\in O(P_F,[f_F])$.
        This implies
        \[
        (\phi^{-1}\alpha\phi')_K=\phi_K^{-1}\alpha_K\phi'_K=\psi^{-1}_K\beta_K\psi'_K=(\psi^{-1}\beta\psi')_K\ .
        \]
        By Lemma~\ref{LM:cart-square-II}(i), there exists an isomorphism $\eta:P''\to P'$ with
        \[
        \eta_S=\phi^{-1}\alpha\phi'\qquad\text{and}\qquad \eta_F=\psi^{-1}\beta\psi'\ .
        \]
        We now have
        \[
        ( [f']\bullet \eta)_S=[f'_S]\bullet (\phi^{-1}\alpha\phi')=
        [f_S]\bullet (\alpha\phi')=
        [f_S]\bullet\phi'=[f''_S]
        \]
        and likewise, $([f']\bullet\eta)_F=[f''_F]$. By Lemma~\ref{LM:cart-square-II}(ii), this means $[f']\bullet\eta=[f'']$,
        so $(P',[f'])\cong (P'',[f''])$.

        Assume now that $O_K\subseteq G_S G_F$. We will show that $\Phi$ is onto, and moreover, that
        every double coset $O_S \eta O_F$ arises from a quadratic form defined on $P$.
        Let $O_S \eta O_F$ be a double coset in $O_K$. Since $O_K\subseteq G_SG_F$, we
        can write $\eta=\alpha_K\beta_K^{-1}$ with $\alpha\in \units{\End_{A_S}(P_S)}$ and
        $\beta\in\units{\End_{A_F}(P_F)}$. Since $\alpha_K\beta^{-1}_K\in O_K$, we have
        \[
        ([f_S] \bullet\alpha)_K=[f_K]\bullet\alpha_K=[f_K]\bullet(\alpha_K\beta^{-1}_K\beta_K)=[f_K]\bullet\beta_K=([f_F] \bullet\beta)_K\ .
        \]
        By Lemma~\ref{LM:cart-square-II}(ii), there exists unique $[g]\in \calQ(P)$ with
        \[
        [g_S]= [f_S] \bullet \alpha\qquad \text{and}\qquad [g_F]=[f_F]\bullet\beta\ .
        \]
        We clearly have $(P,[g])\in\gen_{S,F}(P,[f])$, and $\Phi(P,[g])=O_S\alpha_K\beta_K^{-1}O_F=O_S\eta O_F$ by the definition of $\Phi$,
        as required.
    \end{proof}

    \begin{remark}\label{RM:double-cosets}
		At this level of generality,
		we do not know if the map $\Phi$ is onto.  In \cite[I.\S11]{Kn91}, Knus calls the square \eqref{EQ:rings-cart-square}
        a \emph{patching diagram for quadratic modules} when $\Phi$ is an isomorphism,
        and several other conditions are satisfied.

        One can use descent theory to show that $\Phi$ is an isomorphism in certain cases, even without
        assuming $O_K\subseteq G_S G_F$. This typically requires $S\oplus F$ to be faithfully flat, $K=S\otimes_R F$,
        and one must show that any $\psi\in O_K$, which we view
        as $\psi:[(f_S)_F]\to[(f_F)_S]$  can be completed
        to a \emph{descent data} (i.e.\ a family of isometries $\psi_{T,Z}:[(f_{T})_{Z}]\to
        [(f_{Z})_{T}]$ for  $T,Z\in\{F,S\}$ satisfying the \emph{cocycle condition};
        see \cite[p.~132]{NeronModels}, for instance).
        For certain families
		of squares as in \eqref{EQ:rings-cart-square},
		this problem was studied extensively
		in the literature under the broader context
		of torsors of group schemes.
		See \cite[Pr.~2.6]{CollOjan92} for a
		general result of this kind.
		These methods seem to go back at least as
		far as \cite[Apx.]{FerRay70}.

		We further comment that the double cosets considered
    	in Theorem~\ref{TH:patching-general} are strongly related with Nisnevich cohomology of affine
    	group schemes over Dedekind domains. Specifically, let $R$ be a Dedekind domain, $F$ its fraction field,
    	$S$ the product of its (non-archimedean) completions, and $K$ the ring of (finite) adeles of $R$.
    	Then for any affine group scheme $\mathbf{G}$ over $R$ satisfying certain mild assumptions, the double cosets
        $\mathbf{G}(S)\!\setminus\! \mathbf{G}(K)/\mathbf{G}(F)$
    	are in correspondence with the first Nisnevich cohomology $\mathrm{H}^1_{\mathrm{Nis}}(R,\mathbf{G})$;
    	see \cite[Th.~3.5]{Nisne82} and also \cite[Th.~2.1]{Nisne84}, \cite[Apx.]{Gille02}.
    \end{remark}

    \begin{remark}
    	Theorem~\ref{TH:patching-general}, and also Theorem~\ref{TH:main} below, do not
    	assume that $\uO([f])$ is defined as an affine $R$-scheme. For example, consider the cases
    	\begin{itemize}
    		\item $(A,\sigma,u,\Lambda)=(\smallSMatII{\Z}{\Q}{}{\Z},\smallSMatII{a}{b}{}{d}\mapsto\smallSMatII{d}{b}{}{a},1,\Lambda^{\min}(1))$ or
    		\item $(A,\sigma,u,\Lambda)=(\Z,\id,-1,2\Z)$
    	\end{itemize}
    	with $R=\Z$.
    	Theorem~\ref{TH:patching-general} would apply (with suitable $S$, $F$, $K$), but it is not clear to us whether $\uO([f])$ is
        defined as an affine $R$-scheme.
	\end{remark}

\section{Patching}
\label{section:patching}

    We now state and prove a patching theorem for quadratic spaces.

    \begin{thm}\label{TH:main}
        Let $K$ be a commutative semilocal topological ring,
        let $S,F\subseteq K$ be subrings of $K$,
        let $R=S\cap F$, let $(A,\sigma,u,\Lambda)$ be a unitary
        $R$-algebra, and let $(P,[f])\in\Quad[u,\Lambda]{A,\sigma}$.
        Assume that:
        \begin{enumerate}
            \item[$(0)$] $\F_2$ is not an epimorphic image of $K$,
            \item[$(1)$] $\units{K}$ is open in $K$ and the map $a\mapsto a^{-1}:\units{K}\to \units{K}$
            is continuous,
            \item[$(2)$] $F$ is dense in $K$ and $F\cap\units{K}=\units{F}$,
            \item[$(3)$] $A_F\in\rproj{F}$ and $\Lambda_F$ is a summand of $A_F$,
            \item[$(4)$] $K$ and $A$ are flat $R$-modules,
            \item[$(5)$] $S$ is open in $K$.
        \end{enumerate}
        Let $\calI=\calI(P_K)$ and $\Delta:=\Delta_{\calI,[f_K]}$ (see~\ref{subsection:reflections}, \ref{subsection:dickson-and-transfer}). Then
        \[
        |\gen_{S,F}(P,[f])|=\frac{|\Delta(O([f_K]))|}{\big|\Delta(O([f_F]))+ \Delta(O([f_S]))\big|}\ .
        \]
        In addition, for all $(P',[f'])\in\gen_{S,F}(P,[f])$, we have $P\cong P'$.
    \end{thm}

    Before giving the proof, let us present an example in which the theorem can be applied,
    and the $\{S,F\}$-genus can be given a more concrete meaning.

    \begin{example}\label{EX:basic-scenario}
        Suppose $R$ is a Dedekind domain with finitely many ideals (or equivalently, $R$ is a semilocal PID).
        For $\frakp\in\Spec R$, denote by $\hat{R}_\frakp$ the completion of $R_{\frakp}$ (which
        is a discrete valuation ring) and let $\hat{F}_\frakp$
        denote the fraction field of $\hat{R}_{\frakp}$. Note that $F:=\hat{R}_0$ is just the
        fraction field of $R$.
        We endow $\hat{R}_\frakp$ and $\hat{F}_\frakp$
        with their natural topologies.
        Now, let
        \[S=\prod_{0\neq\frakp}\hat{R}_\frakp\qquad\text{and}\qquad K=\prod_{0\neq \frakp}\hat{F}_\frakp\ .\]
        We embed $F$ diagonally in $K$. It it well-known that $F$ is dense in $K$ \cite[Th.~11.6]{Endler72ValuationTheory},
        and $F\cap S=R$. Furthermore, any torsion-free $R$-module is flat \cite[Th.~4.69]{La99}.
        This means that Theorem~\ref{TH:main} can be applied with any unitary $R$-algebra $(A,\sigma,u,\Lambda)$
        such that $A$ is torsion-free and $\dim_FA_F<\infty$. Moreover, for $(P,[f])\in\Quad[u,\Lambda]{A,\sigma}$,
        we have
        \[
        \gen_{S,F}(P,[f])=\gen_{\{\hat{R}_\frakp\where \frakp\in\Spec R\}}(P,[f])\ .
        \]
        That is, the $\{S,F\}$-genus is the genus considered in the introduction.

        Now, let $k(\frakp)$ denote the residue field at $\frakp$, namely, $R_\frakp/\frakp_{\frakp}$.
        If $A$ is finitely generated and projective as an $R$-module,
        then a standard lifting argument (see
        \cite[Th.~2.2(2)]{QuSchSch79} or \cite[Th.~II.4.6.1]{Kn91}, for instance), implies that
        that $\scalarExt{k(\frakp)}{R}(P,[f])\cong \scalarExt{k(\frakp)}{R}(P',[f'])$ if and only if
        $\scalarExt{\hat{R}_\frakp}{R}(P,[f])\cong \scalarExt{\hat{R}_\frakp}{R}(P',[f'])$. Thus, in this case,
        $
        \gen_{S,F}(P,[f])=\gen_{\{k(\frakp)\where \frakp\in\Spec R\}}(P,[f])
        $.
		
		If in addition $\Lambda$ is a summand
		of $A$, then $\uO([f])$
		is defined as an affine $R$-scheme (see the appendix) and  $\gen_{S,F}(P,[f])$
        corresponds to the first Nisnevich cohomology $\mathrm{H}^1_{\mathrm{Nis}}(R,\uO([f]))$ (cf.\ Remark~\ref{RM:double-cosets}).
    \end{example}

    The proof of Theorem~\ref{TH:main} requires a technical lemma. As in section~\ref{section:density},
    we topologize all objects defined over $K$.

    \begin{lem}\label{LM:open-subobjects}
        Keep the assumptions of Theorem~\ref{TH:main} and let $P'\in\rproj{A}$. Then the maps
        (cf.~\ref{subsection:unitary-algs})
        \begin{align*}
            \scalarExt{K}{S}:&\Hom_{A_S}(P_S,P'_S)\to\Hom_{A_K}(P_K,P'_K)\\
            \scalarExt{K}{S}:&\units{\End_{A_S}(P_S)}\to\units{\End_{A_K}(P_K)}\\
            \scalarExt{K}{S}:&\Lambda_{P_S}\to \Lambda_{P_K}\\
            \scalarExt{K}{S}:&O([f_S])\to O([f_K])
        \end{align*}
        are injective, and their image is open (in the appropriate topology).
    \end{lem}

    \begin{proof}
    	When $A\in\rproj{R}$ and $\Lambda$ is a summand of $A$,  $\uO([f])$ is defined as an affine $R$-scheme
        (see the appendix),
        and then the lemma then follows from
    	\cite[Ex.~2.2]{Conrad}.

        We prove the general case directly:
        By Lemma~\ref{LM:natural-isomorphism}, we may identify
        $\End_{A_T}(P_T,P'_T)$ with $\End_A(P,P')_T$ and $\Hom_{A_T}(P_T,(P_T)^*)$ with $\Hom_A(P,P^*)_T$
        for any $R$-algebra $T$.
        As in the proof of Lemma~\ref{LM:cart-square-II}(i),
        $\End_A(P,P')$ is a flat $R$-module. Thus, $\scalarExt{K}{S}:\Hom_A(P,P')_S\to\Hom_A(P,P')_K$
        is an injection (because the inclusion map $S\to K$ is an injection). The other three maps in the lemma
        are just restrictions of $\scalarExt{K}{S}:\Hom_{A_S}(P_S,P'_S)\to\Hom_{P_K}(P_K,P'_K)$
        in the special cases $P'=P^*$ and $P'=P$, so they are  also injective.

        We identify $\Hom_{A_S}(P_S,P'_S)$ and $\Lambda_{P_S}$
        with their images in $\Hom_{P_K}(P_K,P'_K)$ and\ $\Lambda_{P_K}$, respectively.
        To prove that these images are open, it is enough to
        prove the following general claim: Let $Q$ be a f.g.\ projective $K$-module. Then any $S$-submodule $Q'\subseteq Q$ with $Q'\cdot K=Q$
        is open $Q$. (Note that we have $\Hom_A(P,P')_S\cdot K=\Hom_A(P,P')_K$ by definition, and
        $\Lambda_{P_S}\cdot K=\Lambda_{P_K}$ by Proposition~\ref{PR:scalar-ext-for-quad-forms}.)
        Indeed, $Q'$ contains a finite set $\{x_i\}_{i=1}^t$ generating $Q$. It is well-known that
        there are $\{\psi_i\}_{i=1}^t\subseteq \dual{Q}=\Hom_K(Q,K)$ such that $\{x_i,\psi_i\}_{i=1}^t$ is a dual basis
        for $Q$. (See the proof of Proposition~\ref{PR:point-wise-topology} for the definition;
        construct the $\psi_i$-s by a section of the projection
        $\bigoplus_i K\to Q$ given by $\bigoplus_ik_i\mapsto \sum_i x_ik_i$.)
        Each $\psi_i:Q\to K$ is $K$-linear, hence continuous, and thus
        $\bigcap_{i=1}^t\psi_i^{-1}(S)$ is open (because $S$ is open in $K$).
        However, for all $x\in \bigcap_i\psi_i^{-1}(S)$, we have
        $x=\sum_ix_i(\psi_ix)\in\sum_i x_iS\subseteq Q'$. Thus, $Q'$ contains an open $S$-submodule, so $Q'$
        is open in $Q$.

        Next, we show that $\units{\End_{A_S}(P_S)}$ is open in $\units{\End_{A_K}(P_K)}$:
        Write $E_S=\End_{A_S}(P_S)$ and $E_K=\End_{A_K}(P_K)$.
        By what we have shown above, $E_S$ is open in $E_K$.
        Thus, $E_S\cap\units{E_K}$ is open in $\units{E_K}$. By Proposition~\ref{PR:topological-algebra}(ii),
        $(E_S\cap \units{E_K})^{-1}$ is also open in $\units{E_K}$, hence
        $E_S\cap(E_S\cap\units{E_K})^{-1}$ is open as well. But this set is $\units{E_S}$.

        Finally, $O([f_S])$ is open in $O([f_K])$ because
        $O([f_S])=O([f_K])\cap \units{E_S}$.
    \end{proof}

    \begin{proof}[Proof of Theorem~\ref{TH:main}]
        We shall make use of the following well-known fact about topological groups:
        \begin{itemize}
            \item[($*$)] If $G$ is a topological (multiplicative) group and $X$ is any subset of $G$,
            then $\overline{X}=\bigcap_U XU$, where $U$ ranges over all neighborhoods of $1_G$.
        \end{itemize}
        A proof can be found in \cite[Th.~3.3(3)]{Wa93} (the proof is given for abelian groups but generalizes verbatim to non-abelian groups).
        Note that ($*$) implies that for any neighborhood $U$ of $1_G$, we have $\overline{X}\subseteq XU$.
        We henceforth use the notation of~\ref{NT:patching-groups}.

        \smallskip

        It is clear that the rings $R,S,F,K$ form a cartesian square as in section~\ref{section:double-cosets}.
        This square is onto since $S+F\supseteq\overline{F}=K$  by ($*$).
        That $A$ and $K$ are flat over $R$ holds by assumption.
        Furthermore,
        we claim that $O_K\subseteq G_SG_F$.  Indeed,  by Lemma~\ref{LM:open-subobjects}, $G_S$
        is open in $G_K$, and by Proposition~\ref{PR:density-of-units},
        $G_F$ is dense in $G_K$. Thus, by ($*$), $G_SG_F\supseteq\overline{G_F}=G_K$, as required.
        We may therefore apply Theorem~\ref{TH:patching-general}  to get a one-to-one correspondence
        \[
        \gen_{S,F}(P,[f])\longleftrightarrow O_S\!\setminus\! O_K/O_F\ .
        \]
        We also get that $(P',[f'])\in \gen_{S,F}(P,[f])$ implies $P'\cong P$.

        Let $H_K=\Delta(O_K)$, $H_S=\Delta(O_S)$ and $H_F=\Delta(O_F)$.
        We will prove the theorem by constructing a bijection between $O_S\!\setminus\! O_K/O_F$
        and $H_K/(H_S+H_F)$, whose size is clearly the desired quantity.
        Define $\Psi:O_S\!\setminus\! O_K/O_F\to H_K/(H_S+H_F)$ by
        \[\Psi(O_S \phi O_F)=\Delta(\phi)+H_S+H_F\ .\]
        It is clear that $\Psi$ is well-defined and surjective. To show $\Psi$ is injective,
        it is enough to prove that $\Delta^{-1}(\Delta(\phi)+H_S+H_F)=O_S\phi O_F$.

        Let $\phi\in O_K$ and let $t=\Delta(\phi)$. We claim that
        $ \Delta^{-1}(t)\subseteq O_S\phi O_F$.
        Indeed,
        by Theorem~\ref{TH:density} (whose assumptions hold by conditions (0)--(3)),
        $\ker\Delta\subseteq \overline{O_F}$.
        Thus, $\Delta^{-1}(t)=\phi\cdot\ker\Delta\subseteq \phi\overline{O_F}=\overline{\phi O_F}$
        and $\overline{\phi O_F}\subseteq O_S\phi O_F$
        by ($*$).

        Now let $x\in H_S$ and $y\in H_F$.
        Then there are $\psi\in O_S$, $\eta\in O_F$ with $\Delta(\psi)=x$ and $\Delta(\eta)=y$.
        By the previous paragraph, we have
        $
        \Delta^{-1}(x+t+y)\subseteq O_S\psi\phi \eta O_F=O_S\phi O_F
        $, so
        $\Delta^{-1}(\Delta(\phi)+H_S+H_F)=\bigcup_{x\in H_S,y\in H_y}\Delta^{-1}(x+t+y)\subseteq O_S\phi O_F$.
        Since $O_S\phi O_F\subseteq \Delta^{-1}(\Delta(\phi)+H_S+H_F)$ holds by construction, we are done.
    \end{proof}

    \begin{cor}\label{CR:main-for-fields}
        Let $K$ be a commutative \emph{semisimple}  Hausdorff  topological ring such that
        $\units{K}$ is open in $K$ and $a\mapsto a^{-1}:\units{K}\to \units{K}$ is continuous.
        Let $F$ be a dense topological sub\emph{field} of $K$, let $S$ be an open subring
        of $K$, and let $R=S\cap F$. Let $(A,\sigma,u,\Lambda)$ be a unitary $R$-algebra
        such that $A$ is flat as an $R$-module and $\dim_F(A_F)<\infty$,  let
        $(P,[f])\in\Quad[u,\Lambda]{A,\sigma}$, and let $\calI=\calI(P_K)$
        and $\Delta=\Delta_{\calI,[f_K]}$. Then
        \[
        |\gen_{S,F}(P,[f])|=\frac{2^{|\calI|}}{\big|\Delta(O([f_F]))+ \Delta(O([f_S]))\big|}\ .
        \]
        In addition, for all $(P',[f'])\in\gen_{S,F}(P,[f])$, we have $P'\cong P$.
    \end{cor}

    \begin{proof}
        The ring $A_K$ is artinian and hence semiperfect. Thus,
        by Theorem~\ref{TH:gen-by-reflections} and the proof of Theorem~\ref{TH:density},
        $\Delta$ is onto, namely, $\Delta(O([f_K]))=(\Z/2\Z)^{\calI}$.
        If $\F_2$ is an epimorphic image of $K$, then necessarily $F=\F_2$ (because $F$ is a field). Since $K$ is Hausdorff
        and $F$ is dense in $K$, we must have $R=S=F=K=\F_2$ and the corollary becomes a triviality.
        We may therefore assume that $\F_2$ is not an epimorphic image of $K$.
        The corollary
        then follows from Theorem~\ref{TH:main}, provided we verify conditions (1)--(5).
        Indeed,
        conditions (1) and (5) hold by assumption and conditions (2) and (3) hold since $F$ is a field and $\dim_F A_F<\infty$.
        That $A$ is flat as an $R$-module holds by assumption. To see that $K$
        is flat over $R$, observe that $K$ is flat over $F_0$, the fraction field of $R$,
        and $F_0$ is flat over $R$. This proves condition (4).
        (When $K$ is non-discrete, we actually have $F_0=F$:
        Since $S$ is open in $K$, $R=S\cap F$ is open in $F_0$. Thus, $F_0$ is open in $F$.
        Let $x\in F$.
        Then
        there is a neighborhood $U\subseteq F_0$ of $0$  such that $(x+U)(0+U)\subseteq 0+F_0$.
        Since $F$ is non-discrete, $U\neq\{0\}$, so there are $a,b\in\units{F_0}$ such that $(x+a)b\in F_0$. But this means $x\in F_0$.)
    \end{proof}

    The following example shows that in general $\gen_{S,F}(P,[f])$ can be of size  $2^{|\calI|}$ for arbitrarily
    large $\calI=\calI(P_K)$.

    \begin{example}\label{EX:max-genus-size}
        We use the notation of Example~\ref{EX:basic-scenario}
        and
        assume that $\Char F\neq 2$ and $\F_2$ is not an epimorphic image of $R$.
        Let $\frakp_1,\dots,\frakp_t$ be the maximal ideals of $R$,
        and for convenience, set $S_i=\hat{R}_{\frakp_i}$ and $K_i=\hat{F}_{\frakp_i}$.
        Then $K=\prod_i K_i$ and $S=\prod_i S_i$.

        Fix some $u,v\in \units{R}$ and $0\neq\pi\in \Jac(R)$.
        For a commutative $R$-algebra $T$, let $(u,v)_T$ denote the \emph{quaternion algebra}
        with center $T$ determined by $u$ and $v$, namely,
        $T\Trings{x,y\where x^2=u,y^2=v,xy=-yx}$. Let $\sigma=\sigma_T$ be the involution given by
        \[
        (a+bx+cy+dxy)^\sigma = a+bx+cy-dxy\qquad\forall\, a,b,c,d\in T\ .
        \]
        Define
        \[
        A=R+\pi R x+\pi R y+\pi Rxy\subseteq (u,v)_R
        \]
        and set $\Lambda=\Lambda^{\min}(1)=\{a-a^\sigma\where a\in A\}$.
        Then $(A,\sigma,1,\Lambda)$ is a unitary $R$-algebra.
        Assume further that $(u,v)_F$ is a division algebra and $(u,v)_{K_i}$, which is a central-simple $K_i$-algebra, splits
        for all $i$. We claim that for all $(P,[f])\in\Quad[u,\Lambda]{A,\sigma}$
        with $P\neq 0$, we have $|\calI(P_K)|=t$ and $\gen_{S,F}(P,[f])=2^{|\calI(P_K)|}=2^t$.

        Indeed, we clearly have
        \[(A_K,\sigma_K,1,\Lambda_K)\cong\prod_i((u,v)_{K_i},\sigma_{K_i},1,\Lambda_{K_i})\ .\]
        and each of the factors is split-orthogonal by assumption.
        Furthermore, since $P_F\neq 0$,
        $P_{K_i}\neq 0$ for all $1\leq i\leq t$. It follows that we can identify
        $\calI(P_K)$ with $\{1,\dots,t\}$. Denote by $\Delta_i$ the map $\Delta_{[f_{K_i}]}:O([f_{K_i}])\to\Z/2\Z$.
        By Corollary~\ref{CR:main-for-fields} and the definition of $\Delta_{\calI(P_K),[f_K]}$,
        it is enough to show that $\Delta_i(O([f_F]))=0$ and $\Delta_i(O([f_{S_i}]))=0$ for all $1\leq i\leq t$.

        Let us first show that $\Delta_i(O([f_{S_i}]))=0$.
        We have $A_{S_i}=S_i+\pi S_i x+\pi S_jy+\pi S_i xy$. It is easy
        to check that $\pi\cdot(u,v)_{S_i}\subseteq\Jac(A_{S_i})$. This means $A_{S_i}/\Jac(A_{S_i})$ is isomorphic to an epimorphic image
        of $S_i/\pi S_i$, hence $A_{S_i}$ is local and $A_{S_i}/\Jac(S_i)$ is isomorphic
        to $R/\frakp_i$.
        Thus, $O([f_{S_i}])$ is generated
        by reflections (Theorem~\ref{TH:gen-by-reflections}; $\F_2$ is not an epimorphic image of $R$), so it is  enough to prove
        that $\Delta_i(\phi_{K_i})=0$ for every reflection $\phi$ of $[f_{S_i}]$.
        Indeed, in this case, $\phi_{K_i}$ is a reflection of $[f_{K_i}]$
        and hence $\Delta_i(\phi_{K_i})=\deg A_{K_i}+2\Z=0$ by \cite[Pr.~5.2]{Fi14A}.
        This argument also shows that  $\Delta_i(O(P_F,[f_F]))=0$ (because $A_F=(u,v)_F$ is a division ring, and hence local).
        Alternatively, $\Delta_i(O(P_F,[f_F]))=0$  follows from \cite[Pr.~5.1]{Fi14A}.

        Explicit choices of $R,u,v,\pi$ satisfying all previous conditions
        are given as follows: Take $R$ to be $\Z$ localized at the multiplicative set $\Z\setminus \bigcup_{i=1}^t \Z p_i$
        where  $p_1,\dots,p_t$ are distinct odd prime numbers,
        let $u=v=-1$, and take $\pi=p_1\cdots p_t$.
    \end{example}

    \begin{remark}\label{RM:infinite-rational-genus}
        In general, $\gen_F(P,[f])$ and $\gen_S(P,[f])$
        can be infinite even when $R,S,F,K$ are
        as in Example~\ref{EX:basic-scenario}.
        For example, take $K=\C(s)(\!(t)\!)$ with the $t$-adic topology, $S=\C(s)[[t]]$,
        $F=K(s)(t)$ and $R=S\cap F$.
        Define $A=\smallSMatII{R}{t^2 R}{R}{R}$
        and $\sigma:A\to A$ by $\smallSMatII{x}{t^2y}{z}{w}^\sigma=\smallSMatII{x}{t^2z}{y}{w}$
        ($x,y,z,w\in R$), and let $\Lambda=\Lambda^{\min}(1)$.
        Then $(A,\sigma,1,\Lambda)$
        is a unitary $R$-algebra with $
        A_S=\smallSMatII{S}{t^2S}{S}{S}
        $
        and $A_F=\nMat{F}{2}$. The assumptions of Theorem~\ref{TH:main}
        are easily seen to hold.

        For every $a\in A$ with $a^\sigma =a$, define a quadratic space $(A,[f_a])$
        by
        \[
        \tilde{f}_a(x,y)=x^\sigma ay\ .
        \]
        We claim that both $\gen_F(A,[f_1])$ and $\gen_S(A,[f_1])$
        consist of infinitely many isomorphism classes.

        Indeed, it is easy to check that $[f_a]\cong [f_b]$ if and only if there exists
        $x\in\units{A}$ such that $x^\sigma ax=b$. Since for all $\alpha\in\C$, we have
        \[
        b(\alpha):=\SMatII{1+\alpha^2s^2}{0}{0}{1+\alpha^2s^2}=\SMatII{1}{-\alpha ts}{\alpha t^{-1}s}{1}^\sigma\cdot 1_A\cdot
        \SMatII{1}{-\alpha ts}{\alpha t^{-1}s}{1},
        \]
        it follows that $(A,[f_{b(\alpha)}])\in\gen_F(A,[f_1])$ for all $\alpha\in\C$.
        However, the forms $\{[f_{b(\alpha)}]\where \mathrm{Im}(\alpha)>0\}$ are pairwise non-isomorphic
        over $(A,\sigma,1,\Lambda)$, as can be easily checked by working modulo $\Jac(A)=\smallSMatII{tR}{t^2R}{R}{tR}$.

        Next, observe that for all $\alpha\in \C$, $1+\alpha t$ has a square root in $S$ (substitute
        $\alpha t$ in the Taylor expansion of $\sqrt{1+x}$), but not in $R$. Since
        \[
        c(\alpha):=\SMatII{1+\alpha t}{0}{0}{1}=
        \SMatII{\sqrt{1+\alpha t}}{0}{0}{1}^\sigma\cdot 1_A\cdot
        \SMatII{\sqrt{1+\alpha t}}{0}{0}{1},
        \]
        we have $(A,[f_{c(\alpha)}])\in\gen_S(A,[f_1])$ for all $\alpha\in \C$.
        The forms $\{[f_{c(\alpha)}]\}_{\alpha\in \C}$ are pair-wise non-isomorphic
        because $[f_{c(\alpha)}]\cong [f_{c(\alpha')}]$
        implies $\det(c(\alpha)) \in \det(c(\alpha'))(\units{F})^2$
        and this is possible only if
        $\alpha=\alpha'$.

        We remark that $|\gen_{S,F}(P,[f])|=1$ for all $(P,[f])\in\Quad[1,\Lambda]{A,\sigma}$
        by Theorem~\ref{TH:enough-idems} below. Replacing $\C$ with a larger field allows
        one to construct similar examples in which the $F$-genus and $S$-genus are of arbitrarily
        large cardinalities.
    \end{remark}

\section{Genus of Size~$1$}
\label{section:size-one}

    This section gives several sufficient conditions guaranteeing that
    the $\{S,F\}$-genus of  Theorem~\ref{TH:main}
    and Corollary~\ref{CR:main-for-fields} has size $1$ for all quadratic spaces.
    In \ref{subsection:general-crit}, we assume the general setting of Theorem~\ref{TH:main},
    and
    in \ref{subsection:orders} we specialize to orders over semilocal principal ideal domains.

\subsection{General Criteria}
\label{subsection:general-crit}

    We assume the setting of Theorem~\ref{TH:main}.
    That is, $K$ is a semilocal topological ring, $F$ and $S$
    are subrings of $K$, $R=S\cap F$, $(A,\sigma,u,\Lambda)$ is a unitary $R$-algebra,
    $(P,[f])\in\Quad[u,\Lambda]{A,\sigma}$, and all  assumptions of Theorem~\ref{TH:main} hold.
    We further write
        $(B,\tau,v,\Gamma)=\scalarExt{K}{R}(A,\sigma,u,\Lambda)$, $\quo{B}=B/\Jac(B)$,
        and $(\quo{B},\quo{\tau},\quo{v},\quo{\Gamma})=\prod_{i=1}^t(B_i,\tau_i,v_i,\Gamma_i)$
        where each factor is a simple unitary ring (see~\ref{subsectnio:csa-s}).

    We note that since $A$ is flat over $R$, the map $A\to A_F$ is injective (because the inclusion
    map $R\to F$ is injective). The maps $A\to A_S$, $A\to A_K$, $A_S\to A_K$, $A_F\to A_K$ are likewise injective.

\medskip

    We begin by observing that when $\calI(A_K)=\emptyset$, we  must have $|\gen_{S,F}(P,[f])|=1$
    by Theorem~\ref{TH:main}. This happens, for example,
    when
    the involutions $\tau_1,\dots,\tau_t$ are all of the second kind, in which case
    we say that $\tau=\sigma_K$ is \emph{essentially of the second kind}.
    This condition is  equivalent to the existence of
    $a\in A_K$ such that $a+\Jac(A_K)$ is central in $\quo{A_K}:=A_K/\Jac(A_K)$
    and $a-a^{\sigma_K}\in\units{A_K}$; the easy proof is left to the reader.
    We record our conclusion in the following proposition.

    \begin{prp}\label{PR:second-kind}
        If $\sigma_K$ is essentially of the second kind,
        then $|\gen_{S,F}(P,[f])|=1$
        for all $(P,[f])\in\Quad[u,\Lambda]{A,\sigma}$.
    \end{prp}

    The next theorem is a general criterion that can be applied to particular examples.
    It will be used in the proofs of Theorems~\ref{TH:hereditary-orders} and~\ref{TH:non-unimodular-forms} below.

\medskip

    Recall that two idempotents $e$, $e'$ in a ring $W$ are called \emph{equivalent},
    denoted $e\approx e'$, if $eW\cong e'W$ as right $W$-modules. This is equivalent
    to the existence of $x\in e'We$ and $y\in eWe'$ such that $yx=e$ and $xy=e'$. Recall
    also
    that an idempotent $e$ is called \emph{primitive} if $eWe$ does not contain idempotents
    beside $0$ and $e$. When $W$ is semiperfect
    (see~\ref{subsection:reflections}),
    the latter is equivalent to $eWe$ being local.
    Furthermore, in this case,
    there are finitely many equivalence classes
    of primitive idempotents. More precisely,
    if $\quo{W}:=W/\Jac(W)\cong\prod_{i=1}^t\nMat{D_i}{n_i}$ with each $D_i$ a division ring,
    and if $\veps_i$ is a primitive idempotent in $\nMat{D_i}{n_i}$ (viewed as an idempotent in $\quo{W}$),
    then any family of idempotents $e_1,\dots,e_t\in W$ with $\veps_i=e_i+\Jac(W)$
    is a family of representatives for the equivalence classes.

    \begin{thm}\label{TH:enough-idems}
        Assume $A_S$ is semilocal, $A_K$ is semiperfect, and
        one of the following equivalent conditions hold:
        \begin{enumerate}
            \item[$(1)$] For every primitive idempotent $e\in A_K$,
            there is an idempotent $e'\in A_S$ with $e\approx e'$.
            \item[$(2)$] There are idempotents $e_1,\dots,e_s\in A_S$
            such that $A_K=\sum_j A_Ke_jA_K$ and $e_jA_Ke_j$ is local for all $j$.
        \end{enumerate}
        Then $|\gen_{S,F}(P,[f])|=1$ for all $(P,[f])\in\Quad[u,\Lambda]{A,\sigma}$.
    \end{thm}

    \begin{proof}
    	Recall that $B=A_K$.
        We first prove that conditions (1) and (2) are equivalent.
        Suppose (1) holds. We can write $1_B=\sum_{j=1}^se_j$ where $e_1,\dots,e_s$ are  primitive idempotents in $B$.
        We clearly have $B=\sum_j Be_j B$ and $e_jBe_j$ is local for all $j$.
        For each $j$, choose $e'_j\in A_S$ such that $e'_j\approx e_j$. It is easy to check that $e_j\approx e'_j$
        implies $Be'_jB=Be_jB$ and $e'_jBe'_j\cong e_jBe_j$. Thus, the idempotents $e'_1,\dots,e'_s$ satisfy
        condition (2). Conversely, assume (2) holds and let $e\in B$ be a primitive
        idempotent. The assumption
        $\sum_jBe_jB=B$ is easily seen to imply that there is a projection
        $\bigoplus_j(e_jB)^{n_j}\to B$ for some $n_1,\dots,n_s\in \N$, and hence
        there is a projection $\bigoplus_j(e_jB)^{n_j}\to eB$. Let $J=\Jac(B)$.
        Then there is some $j$ for which there is a homomorphism $p:e_jB\to eB$
        whose image is not contained in $eJ$. Since $eBe$ is local, $\quo{e}:=e+J$
        is primitive in $\quo{B}=B/J$, and since $\quo{B}$ is semisimple, $eB/eJ\cong \quo{e}(B/J)$ is simple.
        Thus, $\im(p)+eJ=eB$ (because $\im(p)\nsubseteq eJ$), and by Nakayama's Lemma we get $\im(p)=eB$.
        Since $eB$ is projective, this means that $eB$ is isomorphic to a summand of $e_jB$.
        But $\End_B(e_jB)\cong e_jBe_j$ is local, so the only summands of $e_jB$ are $0$ and $e_jB$.
        Thus, $eB\cong e_jB$ and we have $e\approx e_j\in A_S$.

\smallskip

        We now prove that condition (1) implies  $|\gen_{S,F}(P,[f])|=1$.
        We may assume $P\neq 0$. Let $(Q,[g])=(P_K,[f_K])$.
        As in~\ref{subsection:reflections}, $(Q,g)$
        induces quadratic spaces $(Q_i,[g_i])\in\Quad[v_i,\Gamma_i]{B_i,\tau_i}$
        for all $1\leq i\leq t$.
        Let $\calI=\calI(P_K)$.
        By Theorem~\ref{TH:main}, it is enough to prove that
        $\Delta_{\calI,[g]}(O(P_S,[f_S]))=(\Z/2\Z)^{\calI}$.
        By Proposition~\ref{PR:compatiability-with-Dickson-inv}, we may apply transfer
        to assume that $P$ is free, and hence so is $Q$.

\smallskip

        We proceed by recalling some results and notions from \cite{Fi14A}:
        Suppose that we are given
        an idempotent $e\in B$, an element $y\in Qe$, and $c\in \hat{g}(y)\cap e^\sigma Be=(\tilde{g}(y,y)+\Gamma)\cap e^\sigma Be$
        such that multiplication on the left by $c$ induces an isomorphism
        $e B\to e^\sigma B$. The inverse of
        $x\mapsto cx:eB\to e^\sigma B$ is given by left multiplication
        with a unique element of $eBe^\sigma$, which we denote by $c^\circ$.
        In this setting, the map $s_{y,e,c}:Q\to Q$ given by
        \[
        s_{y,e,c}(x)=x-yc^\circ \tilde{h}_g(y,x)\qquad\forall x\in Q
        \]
        is called an \emph{$e$-reflection} of $[g]$. It is always an isometry.
        If $e'\in A$ is another idempotent with $e\approx e'$, then every $e$-reflection is an $e'$-reflection
        (for different $y$ and $c$).
        Furthermore, if $\phi$ is an $e$-reflection, then the induced isometry $\quo{\phi}\in O(\quo{Q},[\quo{g}])$ is an
        $\quo{e}$-reflection, and every $\quo{e}$-reflection is obtained in this way.
        Finally, we record that $\quo{e}$-reflections exist when $\quo{e}$ is primitive and $\quo{Pe}\neq 0$.
        See \cite[\S3, \S5.2]{Fi14A} for proofs.

\smallskip

        We now return to the proof: Let $i\in\calI$. Then $B_i\cong\nMat{L}{n}$ for some field
        $L$ and $n\in\N$. Let $\veps\in \nMat{L}{n}\cong B_i$ denote the idempotent matrix whose $(1,1)$-entry
        is $1$ and its other entries are $0$.
        Since $B$ is semiperfect, there is an idempotent
        $e\in B$ whose image in $\overline{B}$ is $\veps$.
        The idempotent
        $e$ is primitive, so by condition (1) there is an idempotent
        $e'\in A_S$ with $e'\approx e$.

        Suppose  $[f_S]$ has an $e'$-reflection $\phi$.
        Then $\phi_K$ is an $e'$-reflection of $[g]=[f_K]$,
        and
        since $e\approx e'$,  ${\phi_K}$ is
        also
        an $e$-reflection. This implies that the induced isometry $\phi_i\in O([g_i])$
        is an $\veps$-reflection, while $\phi_j=\id_{Q_j}$ for all $j\in\calI\setminus\{i\}$.
        By \cite[Pr.~5.2]{Fi14A}, $\Delta_{[g_i]}(\phi_i)=1+2\Z$, so
        $\Delta_{\calI,[g]}(\phi)=(\delta_{ij}+2\Z)_{j\in\calI}$ where $\delta_{ij}=1$ if $i=j$ and $\delta_{ij}=0$ otherwise.

        We now claim that $[f_S]$ has an $e'$-reflection and hence $(\delta_{ij}+2\Z)_{j\in\calI}\in \Delta_{\calI,[g]}(O([f_S]))$.
        Let $\check{\phantom{a}}$ denote reduction modulo $\Jac(A_S)$ (whenever it makes sense).
        Then it is enough to prove that $(\check{P}_S,[\check{f}_S])$ has an $\check{e}'$-reflection.
        Write $\check{e}'$ as a sum of orthogonal primitive idempotents
        in $\check{A}_S$, $\check{e}'=\sum_je_j$. It is enough to prove that $(\check{P}_S,[\check{f}_S])$ has
        an $e_j$-reflection for all $j$ because their product will be an $\check{e}'$-reflection by \cite[Lm.~3.4]{Fi14A}.
        Since $P$ is free and nonzero, $\check{P}_S$ is free and nonzero,
        and hence $\check{P}_Se_j\neq 0$ for all $j$. This implies that $e_j$-reflections exist, as required.

        Letting $i$ range over $\calI$,
        we have shown that $\Delta_{\calI,[g]}(O([f_S]))$ contains $\{(\delta_{ij}+2\Z)_{j\in\calI}\}_{i\in\calI}$.
        This set spans $(\Z/2\Z)^{\calI}$ as a group, so we are done.
    \end{proof}

    \begin{example}\label{EX:genus-one-order}
        Let $R,S=\prod_{i=1}^tS_i,F,K=\prod_{i=1}^tK_i$ be as in Example~\ref{EX:max-genus-size}. Let $\fraka$ be a
        nonzero  ideal of
        $R$, let
        \[
        A=\SMatII{R}{\fraka}{\fraka}{R}\ ,
        \]
        let $\sigma$ be the transpose involution and let $\Lambda=\{a-a^\sigma\where a\in A\}$.
        Then $(A,\sigma,1,\Lambda)$ is a unitary ring and
        \[
        B=A_K=\prod_{i=1}^t\nMat{K_i}{2},\qquad A_S=\prod_{i=1}^t\SMatII{S_i}{\fraka S_i}{\fraka S_i}{S_i}\ .
        \]
        Let $e_i$
        denote the matrix unit $e_{11}$ of $\nMat{K_i}{2}$, viewed as an element of $A_S$.
        It is easy to see that $e_iBe_i$ is local for all $i$ and $\sum_{i=1}^t Be_iB=B$.
        Thus, by Theorem~\ref{TH:enough-idems}, $|\gen_{S,F}(P,[f])|=1$
        for all $(P,[f])\in\Quad[u,\Lambda]{A,\sigma}$. (Despite this, as in Example~\ref{EX:max-genus-size},
        we have $\calI(P_K)=\{1,\dots,t\}$
        for all $0\neq P\in\rproj{A}$.)
    \end{example}

    \begin{example}
        Let $R,S,F,K$ and $(A,\sigma,u,\Lambda)$
        be as in Example~\ref{EX:max-genus-size}. Then $A_K$ is semiperfect (because it is artinian)
        and $A_S$ is semilocal, but  conditions (1) and (2) of Theorem~\ref{TH:enough-idems}
        do not hold for $A$. Indeed, if this was not so, then the computation of $|\gen_{S,F}(P,[f])|$ in Example~\ref{EX:max-genus-size}
        would contradict the conclusion of Theorem~\ref{TH:enough-idems}.

        The failure of (1) and (2) can also be checked directly:
        In the notation of Example~\ref{EX:max-genus-size}, the ring $A_{S_i}$ is local for all $1\leq i\leq t$,
        hence the only nonzero idempotent in $A_{S_i}$ is its unity, denoted $e_i$. Since $A_S=\prod_{i=1}^tA_{S_i}$,
        any idempotent $e\in A_S$ can be written as $\sum_{i\in I}e_i$ for  $I\subseteq \{1,\dots,t\}$,
        in which case $eA_Ke\cong \bigoplus_{i\in I} A_{K_i}=\prod_{i\in I}\nMat{K_i}{2}$.
        It follows that $eA_Ke$ is not local for any idempotent $0\neq e\in A_S$,
        so (2) is false. The failure of (1) can be checked by similar means.
    \end{example}

\subsection{Orders}
\label{subsection:orders}

    Assume henceforth that $R$ is a semilocal PID with maximal ideals
    $\frakp_1,\dots,\frakp_t$ and fraction field $F$. As in Example~\ref{EX:max-genus-size},
    we set
    $S=\prod_{i=1}^t \hat{R}_{\frakp_i}$, $K=\prod_{i=1}^t\hat{F}_{\frakp_i}$ and $\hat{R}_0=F$.
    We further assume that $R\neq F$.

\medskip

    Recall that an \emph{$R$-order} in an $F$-algebra $E$ is an $R$-subalgebra $A\subseteq E$ which is finitely
    generated as an $R$-module and contains an $F$-basis of $E$.
    Since $R$ is a  PID, an $R$-algebra $A$ is an $R$-order in some $F$-algebra (necessarily isomorphic to $A_F$)
    if and only if $A$ is a finitely
    generated free $R$-module.

    An $R$-order $A$ is called \emph{hereditary} if  its one-sided ideals
    are projective as $A$-modules.
    For example, $R$ itself is a (two-sided) hereditary because it is a PID.
    The structure of hereditary $R$-orders is  well-understood and we refer the reader
    to \cite[Ch.~9]{MaximalOrders} for further details.

     \begin{example}
        (i) An $R$-order $A$ is \emph{maximal} if it is not properly
        contained in any $R$-order inside $A_F$. It is well-known that
        maximal $R$-orders are hereditary \cite[Cor.~1.6.0]{HijNish94}.

        (ii) Certain \emph{weak crossed products} over $R$ (i.e.\ crossed products defined using $2$-cocycles taking non-invertible values)
        are hereditary (or \emph{semihereditary}); see \cite{Kauta14} and related papers.

        (iii) Let $\pi\in R$ be a prime element. Then
        \[
        A=\SMatII{R}{\pi R}{R}{R}
        \]
        is a non-maximal  hereditary order.
        This can be checked directly, or by using the results in \cite[Ch.~9]{MaximalOrders}.

        (iv) The $R$-order of Example~\ref{RM:infinite-rational-genus} is not hereditary.
        The $R$-order $A$ of Example~\ref{EX:genus-one-order} is not hereditary
        unless $\fraka=R$.
    \end{example}

    We shall make extensive usage of the following theorem.

	\begin{thm}\label{TH:local-hereditarity}
        Let $A$ be an $R$-order. Then
        $A$ is hereditary if and only if $A_{\hat{R}_{\frakp}}$ is hereditary for all
        $0\neq\frakp\in\Spec R$ (i.e.\ if $A_S$ is hereditary).
        In this case, $A_F$ and
        $A_{\hat{F}_\frakp}$ are semisimple for all  $0\neq\frakp\in\Spec R$. 
    \end{thm}

    \begin{proof}
        By \cite[p.~8]{AusGold60B}, the order $A$ is hereditary
    	if and only if $A_{R_\frakp}$ is hereditary for all $0\neq \frakp\in\Spec R$.
        We may therefore assume $R$ is a DVR. Write $\hat{R}$ and $\hat{F}$ for the completions
        of $R$ and $F$.
        Recall that the order $A$ is called \emph{extremal} if there is no order $A\subsetneq A' \subseteq A_F$
        with $\Jac(A)\subseteq\Jac(A')$. By \cite[Th.~1.6]{HijNish94}, the hereditary orders
        are precisely the extremal orders.
        Observe that tensoring with $\hat{R}$ induces a bijection between the full $R$-lattices
        in $A_F$ (i.e.\ finitely generated $R$-modules containing an $F$-basis of $A_F$) and the full $\hat{R}$-lattices
        in $A_{\hat{F}}$; see \cite[Th.~5.2(ii)]{MaximalOrders}, for instance.
        This bijection clearly takes $R$-orders to $\hat{R}$-orders, maximal left ideals to maximal left ideals, and hence the Jacobson
        radical of an $R$-order $A'\subseteq A_F$ to the Jacobson radical of $A'_{\hat{R}}$. As a result, $A$ is extremal if
        and only if $A_{\hat{R}}$ is extremal. This proves the first statement of the theorem.\footnote{
    		This statement also follows
            from
    		\cite[Th.~3.30]{MaximalOrders}. However, the proof there
            seems to have a gap because it assumes that all $A_{\hat{R}_\frakp}$-modules
            are obtained from $A$-modules via scalar extension, which is false in general.
        }

        That $A_F$ and $A_{\hat{F}_\frakp}$ are semisimple follows from \cite[Th.~1.7.1]{HijNish94}, stating that any
        finite-dimensional $F$-algebra
        with a hereditary $R$-order is semisimple.
    \end{proof}

    The following theorem is a special case of Theorem~\ref{TH:enough-idems}.

    \begin{thm}\label{TH:hereditary-orders}
        Let $(A,\sigma,u,\Lambda)$ be a unitary $R$-algebra
        such that $A$ is a hereditary $R$-order.
        Then $|\gen_{\{F,S\}}(P,[f])|=1$ for all $(P,[f])\in\Quad[u,\Lambda]{A,\sigma}$.
    \end{thm}

    \begin{proof}
        It is clear that $A_K$ is semiperfect,
        $A_S$ is semilocal, and conditions (1)--(5) of Theorem~\ref{TH:main}
        are satisfied. As in Corollary~\ref{CR:main-for-fields}, we can ignore
        condition (0) of Theorem~\ref{TH:main}.
        It is therefore enough to verify condition (1) of Theorem~\ref{TH:enough-idems}.

        By Theorem~\ref{TH:local-hereditarity}, $A_S$ is hereditary.
        Let $e$ be  a primitive idempotent in $A_K$, and let $I=A_S\cap eA_K$.
        Define a right $A_S$-module
        homomorphism  $\psi:A_S/I\to (1-e)A_K$ by $\psi(a+I)= (1-e)x$ (this is well-defined
        because $(1-e)I=0$). We claim that $\psi$
        is injective. Indeed, if $\psi(a+I)=0$, then $(1-e)a=0$ and hence $a=ea$, which
        means $a\in I$. Since $A_S/I$ is a cyclic $A_S$-module and $K=F\cdot S$, the image of $\psi$
        is contained in $tA_S$ for some $t\in F$.
        Therefore,
        $A_S/I$ is isomorphic to a right ideal of $A_S$. Since $A_S$ is hereditary,
        $A_S/I$ is projective, so the exact sequence $0\to I\to A_S\to A_S/I\to 0$
        splits. We therefore have $A_S=I\oplus J$ where $J$ is some right ideal of $A_S$.
        Write $1=e'+a$ with $e'\in I$ and $a\in J$. It is easy to check that $e'$ is an idempotent
        satisfying $e'A_S=I$ and $e'A_K=eA_K$, so $e'\approx e$.
    \end{proof}

    \begin{cor}
        Keep the assumptions of Theorem~\ref{TH:hereditary-orders}
        and let $(P,[f]),(P',[f'])\in\Quad[u,\Lambda]{A,\sigma}$.
        For all $\frakp\in \Spec R$, let $k(\frakp)$ denote the fraction
        field of $R/\frakp$.
        Then $(P,[f])\cong (P',[f'])$ if and only if $(P_{k(\frakp)},[f_{k(\frakp)}])\cong (P'_{k(\frakp)},[f'_{k(\frakp)}])$
        for all $\frakp\in \Spec R$ (including $\frakp=0$).
    \end{cor}

    \begin{proof}
        We have $A\in\rproj{R}$ by assumption.
        Thus, as noted in Example~\ref{EX:basic-scenario},
        $\gen_{S,F}(P,[f])=\gen_{\{k(\frakp)\where \frakp\in\Spec R\}}(P,[f])$.
        The corollary therefore follows from Theorem~\ref{TH:hereditary-orders}.
    \end{proof}

    \begin{remark}
        Theorem~\ref{TH:hereditary-orders} actually
        holds in the general setting of Theorem~\ref{TH:main}
        if one assumes that $A_K$ is semiperfect, $A_S$ is semilocal and \emph{semi}hereditary, and  $K=U^{-1}S$
        where $U=\units{K}\cap S$. The proof is essentially the same.
    \end{remark}

\section{Further Quadratic Objects}
\label{section:further-quadratic-objects}

    We keep the notation of~\ref{subsection:orders}.
    In this final section, we extend Theorems~\ref{TH:main} to systems of sesquilinear
    forms and Theorem~\ref{TH:hereditary-orders} to \emph{non-unimodular} hermitian forms.
    This will be done using results from \cite{BayerFain96}, \cite{BayerMold12}
    and \cite{BayFiMol13}. As an application, we prove that Witt's Cancellation Theorem
    and Springer's Theorem hold for hermitian forms over hereditary orders.

    \emph{In contrast to the previous sections, the results of this section require that $2\in\units{R}$.}

\subsection{Systems of Sesquilinear Forms}
\label{subsection:systems}

    Let $A$ be an $R$-order (see~\ref{subsection:orders})
    and let $\{\sigma_i\}_{i\in I}$ be a family of $R$-involutions on $A$.
    Then each involution $\sigma_i$ induces a duality
    $*_i:\rproj{A}\to \rproj{A}$ as in~\ref{subsection:herm-forms}.
    A \emph{system of sesquilinear forms} over $(A,\{\sigma_i\}_{i\in I})$
    is a pair $(P,\{f_i\}_{i\in I})$
    such that $(P,f_i)$ is a sesquilinear form over $(A,\sigma_i)$ for all $i\in I$.
    If $(P',\{f'_i\}_{i\in I})$ is another system, then
    an isometry from $(P,\{f_i\})$ to $(P',\{f'_i\})$ is
    an isomorphism of $A$-modules $\phi:P\to P'$
    such that $f_i=\phi^{*_i}f'_i\phi$ for all $i\in I$.

    The notion of genus naturally extends to systems of sesquilinear forms over
    $(A,\{\sigma_i\}_{i\in I})$. For brevity, we let
    \[
    \gen(P,\{f_i\})=\gen_{\{\hat{R}_{\frakp}\where \frakp\in\Spec R\}}(P,\{f_i\})
    \]
    where $\hat{R}_0=F$.

    \begin{thm}\label{TH:main-for-systems}
        Assume $2\in \units{R}$. Then,
        in the previous setting, $|\gen(P,\{f_i\})|$ is  a finite power of $2$.
    \end{thm}

    The proof uses the language of \emph{hermitian categories}. Our notation will
    follow
    \cite[\S2,\S4]{BayFiMol13} and we refer the reader to this source for definitions;
    see also \cite{QuSchSch79}, \cite[Ch.~7]{SchQuadraticAndHermitianForms}, \cite[Ch.~II]{Kn91} or \cite{Bal05}.

    \begin{proof}
        By \cite[Th.~4.1]{BayFiMol13}, there is an isomorphism between
        the category
        of systems sesquilinear forms over $(A,\{\sigma_i\})$ and the category
        of unimodular $1$-hermitian forms over a certain hermitian category, which we denote by $\catD$
        (in \cite{BayFiMol13}, this category is denoted by $\TDA[I]{\catC}$ where $\catC=\rproj{A}$).
        Furthermore, this isomorphism is compatible with flat scalar extension over $R$ (\cite[Cr.~4.4]{BayFiMol13}
        and the comment before it).
        It is therefore enough to prove the theorem for \emph{unimodular} $1$-hermitian forms over the hermitian
        category $\catD$.

        Let $(Q,g)$ and $(Q',g')$ be two $1$-hermitian forms over $\catD$ with the same genus.
        We claim that $Q\cong Q'$.
        Indeed, by applying transfer in hermitian categories (see \cite[\S2C, \S2E]{BayFiMol13}),
        we may assume $(Q,g)$ and $(Q',g')$ are $1$-hermitian forms over a ring with involution $(A',\sigma')$
        where $A'$ is the endomorphism ring of some object in $\catD$.
        By the construction of $\catD$, $A'$ an $R$-subalgebra of $\End_{A}(P_1)\times\End_{A}(P_2)^\op$
        for some $P_1,P_2\in\rproj{A}$ (see \cite[\S4]{BayFiMol13}),
        and hence an $R$-order. In addition, since the hermitian structure on $\catD$
        is $R$-linear, $\sigma'$ is  an $R$-involution.
        Now, since $2\in\units{R}$, $1$-hermitian forms are equivalent to quadratic
        spaces over $(A',\sigma',1,\Lambda^{\max}(1))$ (Remark~\ref{RM:two-is-invertible}).
        Thus, by Corollary~\ref{CR:main-for-fields}, $Q'\cong Q$.

        This implies that for any $(Q',g')\in\gen(Q,g)$, we have $Q'\in\catD|_Q$ (see \cite[\S2C]{BayFiMol13}
        for the definition). In particular, $\gen(Q,g)$ is contained
        in the category of unimodular $1$-hermitian forms over $\catD|_Q$. The transfer functor with respect
        to $(Q,g)$ induces
        an equivalence between this category and the category of unimodular
        $1$-hermitian forms over $(A'',\sigma'')$, where $A''=\End_{\catD}(Q)$
        and $\sigma''$ is induced from $g$.  As before, $(A'',\sigma'')$ is an
        $R$-order with an $R$-involution, so we are reduced to the case where $(Q,g)$ is a $1$-hermitian
        form over an $R$-order with an $R$-involution. Since $2\in\units{R}$,
        we may apply Corollary~\ref{CR:main-for-fields} and conclude that
        $\gen(Q,g)$ is a finite power of $2$.
    \end{proof}

    One can guarantee that $|\gen(P,\{f_i\})|=1$
    in case all the involutions $\{\sigma_i\}_{i\in I}$ restrict to a particular
    involution on $\Cent(A)$ which is  of the second kind:
    Let $R'/R$ be a Galois extension (of commutative
    rings) with Galois group $\{\id,\tau\}$ ($\tau\neq\id$); see \cite[Apx.]{AusGold60A} or \cite{Sa99}
    for the general definition. In our case, this  is equivalent to
    $R'\cong R[x\,|\, x^2=ax+b]$ for $a,b\in R$ satisfying $a^2+4b\in\units{R}$, in which case $\tau$ is
    the $R$-automorphism of $R'$ sending $x$ to $a-x$.

    \begin{thm}\label{TH:second-kind-systems}
        Keep the previous setting and assume in addition that $A$ is an $R'$-algebra
        and $\sigma_i|_{R'}=\tau$ for all $i\in I$. Then $|\gen(P,\{f_i\})|=1$
        for every system of sesquilinear forms $(P,\{f_i\})$ over $(A,\{\sigma_i\})$.
    \end{thm}

    \begin{proof}
        Let $(A'',\sigma'')$ be the $R$-order with involution constructed
        in the last paragraph of the proof of Theorem~\ref{TH:main-for-systems}.
        We claim that $\sigma''_K$ is essentially of the second kind (see~\ref{subsection:general-crit}),
        in which case we are done by Proposition~\ref{PR:second-kind}.
        By the construction of the hermitian category $\catD$,
        $A''$ is an $R'$-algebra and $\sigma''|_{R'}=\tau$ (see \cite[\S5F]{BayFiMol13}).
        Since $R'/R$ is Galois with Galois group $\{\id,\tau\}$,
        $R'_F/F$ is Galois with Galois group $\{\id,\tau_F\}$. In this case,
        either $R'_F\cong F\times F$ or $R'_F/F$ is a $2$-dimensional field
        extension. In either case,  it is easy to see that there is $a\in R'_F\subseteq\Cent(A_K)$
        such that $a-a^{\tau_F}\in\units{(R'_F)}$, and hence $\sigma''_K$ is essentially of the second kind.
    \end{proof}

    \begin{remark}
        Theorem~\ref{TH:main-for-systems}
        also holds for systems of sesquilinear forms over \emph{$R$-categories}
        (see \cite[\S2D, \S4]{BayFiMol13})
        in which the $\Hom$-sets are finitely generated projective $R$-modules.
        The proof is similar.
    \end{remark}

\subsection{Non-Unimodular Hermitian Forms}
\label{subsection:non-unimodular}

    Let $(A,\sigma)$ be an $R$-order with an $R$-involution and let $u\in \Cent(A)$
    be an element such that $u^\sigma u=1$.
    We extend the notion of genus to (not-necessarily unimodular) $u$-hermitian forms in the obvious way.
    As  in~\ref{subsection:systems}, for every $u$-hermitian form $(P,f)$ over $(A,\sigma)$, we set
    \[
    \gen(P,f)=\gen_{\{\hat{R}_\frakp\where \frakp\in\Spec R\}}(P,f)
    \]
    for brevity.
    (Note that if $(P,f)$ is unimodular and $(P',f')\in\gen(P,f)$,
    then $f'$ is also unimodular by virtue of Corollary~\ref{CR:unimodularity-in-genus}.)

    Viewing $(P,f)$ as a system of sesquilinear forms (consisting of just one form),
    Theorem~\ref{TH:main-for-systems} implies that $|\gen(P,f)|$ is a finite power of $2$.
    We now strengthen this assertion by showing that $|\gen(P,f)|=1$ when $A$ is hereditary
    (see~\ref{subsection:orders}).

    \begin{thm}\label{TH:non-unimodular-forms}
        Keep the previous setting and assume $A$ is hereditary
        and $2\in\units{R}$. Then $|\gen(P,f)|=1$ for any
        $u$-hermitian form $(P,f)$ over $(A,\sigma)$.
    \end{thm}

    We first prove the following lemma.
    We shall use the fact that a submodule of a projective
    module over a hereditary ring is also projective
    \cite[Cr.~2.26]{La99}.

    \begin{lem}\label{LM:morphism-category}
        Let $W$ be a hereditary flat $R$-algebra
        such that $W_F:=W\otimes_R F$ is artinian.
        Let $\Mor(\rproj{W})$ denote the category
        of morphisms in $\rproj{W}$. Then
        every object in $\Mor(\rproj{W})$ is a direct
        sum of objects $M$ such that $\End_{\Mor(\rproj{W})}(M)\otimes_RF$
        is local.
    \end{lem}

    \begin{proof}
        Recall that the objects of $\Mor(\rproj{W})$
        are triples $(U,g,V)$
        where $U,V\in \rproj{W}$ and $g\in\Hom_W(V,U)$. A morphism from
        $(U,g,V)$ to $(U',g',V')$ is a pair $(\phi,\psi)\in\Hom_W(U,U')\times\Hom_W(V,V')$
        such that $\psi g=g\phi$. It is easy to check that for every \emph{flat} $R$-algebra $T$, there is a
        canonical isomorphism
        \[\End_{\Mor(\rproj{W_T})}(U_T,g_T,V_T)\cong \End_{\Mor(\rproj{W})}(U,g,V)\otimes_R T\ .\]
        We may therefore identify $\End(U_F,g_F,V_F)$
        with $\End(U,g,V)\otimes_R F$. The flatness of $W$ over $R$ alows
        us to consider $U$ and $V$ as  submodules of $U_F$ and $V_F$, respectively.

        Let $e_1,\dots,e_t$ be a complete list of primitive idempotents in $W_F$ up to equivalence
        (see~\ref{subsection:general-crit}). Since $W_F$ is artinian, it is well-known
        that every projective $W_F$-module is a direct sum of copies of $\{e_iW_F\}_{i=1}^t$.
        We now claim that every $V\in\rproj{W}$ is a direct sum of modules $U$
        such that $U_F\cong e_iW_F$ for some $i$.
        This is clear if $V=0$. Otherwise, there is $1\leq i\leq t$
        and a projection of $W_F$-modules $p:V_F\to e_iW_F$. The $W$-module
        $p(V)$ is contained in $t\cdot e_iW$ for some $t\in\units{F}$,
        and hence it is projective (because $W$ is hereditary).
        Thus, we can write $V=V_1\oplus \ker (p|_V)$. It is clear that $(V_1)_F\cong e_iW_F$ via $p$.
        Now proceed by induction on $\mathrm{length}(V_F)$.

        Let $(U,g,V)\in\Mor(\rproj{W})$ be an indecomposable object.
        We claim that  $\End(U,g,V)_F$ is local. If $V=0$, then $g=0$
        and $U$ is necessarily indecomposable in $\rproj{W}$.
        By the previous paragraph, there is some $1\leq i\leq t$
        such that $U_F\cong e_iW_F$.
        Thus, $\End(U,g,V)_F\cong \End_{W}(U)_F\cong \End_{W_F}(e_iW_F)\cong
        e_iW_Fe_i$, which is local.
        We may therefore assume $V\neq 0$.
        Now, by the previous paragraph, we can write
        $V=V_1\oplus V_2$ where $(V_1)_F\cong e_iW_F$ for some $i$.
        Let $p$ denote the projection from $V$ onto the summand $V_1$.
        Then $p(g(U))$ is projective (because $W$ is hereditary),
        hence we can write $U=U_1\oplus U_2$ with $U_2=\ker (p g)$.
        It is easy to see that $(U,g,V)=(U_1,g|_{U_1},V_1)\oplus (U_2,g|_{U_2},V_2)$.
        Since $(U,g,V)$ was assumed to be indecomposable, $U=U_1$ and $V=V_1$.
        It also follows that $g$ is injective, so we may identify
        $U$ as a submodule of $V$ via $g$. The ring $\End(U,g,V)_F=\End(U_F,g_F,V_F)$
        therefore consists of those elements $\phi\in\End_{W_F}(V_F)\cong e_iW_Fe_i$
        such that $\phi(U_F)\subseteq U_F$.
        Since $e_iW_Fe_i$ is local, at least one of $\phi$, $1-\phi$ is invertible
        in $\End_{W_F}(V_F)$, say it is $\phi$.
        Since $\mathrm{length}(U_F)$ is finite, we must have $\phi(U_F)=U_F$
        and hence $\phi^{-1}(U_F)=U_F$, which implies $\phi^{-1}\in \End(U_F,g_F,V_F)$.
        This shows that for all $\phi\in\End(U,g,V)_F$, at least one of $\phi$, $1-\phi$
        is invertible, so  $\End(U,g,V)_F$ is local.

        Finally, we note that every object $(U,g,V)\in\Mor(\rproj{W})$
        can be written as a finite sum of indecomposable objects because
        every non-trivial decomposition of $(U,g,V)$ induces a non-trivial decomposition
        of $(U_F,g_F,V_F)\in\Mor(\rproj{W})$ and then length
        of $(U_F,g_F,V_F)$ is finite since $W_F$ is artinian.
    \end{proof}

    \begin{proof}[Proof of Theorem~\ref{TH:non-unimodular-forms}]
        We follow the same argument as in the proof of Theorem~\ref{TH:main-for-systems},
        but use the hermitian category constructed in
        \cite[\S3]{BayerFain96} instead of the one in \cite[\S4]{BayFiMol13}.\footnote{
            One can  also take the hermitian categories constructed in \cite[\S3]{BayFiMol13}
            and \cite[\S3]{BayerMold12}. The hermitian category of  \cite[\S3]{BayerFain96} embeds
            as a full subcategory inside these hermitian categories, which are
            isomorphic in our setting \cite[Rm.~3.2]{BayFiMol13}.
        } This category is
        $\Mor(\rproj{A})$ endowed with a certain hermitian structure.
        By \cite[Th.~1]{BayerFain96},
        there is an equivalence between the category of (arbitrary) hermitian forms over
        $(A,\sigma)$ and the category of unimodular $1$-hermitian forms over $\Mor(\rproj{A})$.
        One can check that this correspondence is compatible with flat scalar extension in the sense
        of \cite[\S2.4]{BayFiMol13}, either
        directly, or by arguing as in the proof of \cite[Pr.~3.7]{BayFiMol13}.

        Arguing as in the proof of Theorem~\ref{TH:main-for-systems}, we are reduced
        to show that the genus of unimodular $1$-hermitian forms over $(A'',\sigma'')$
        has size $1$, where $A''$ is the endomorphism ring of some object in $\Mor(\rproj{A})$.
        It is therefore enough to show that the endomorphism ring of every object in $\Mor(\rproj{A})$
        satisfies the assumptions of Theorem~\ref{TH:enough-idems}.

        Let $(U,g,V)\in\Mor(\rproj{A})$.
        As observed in the proof of Lemma~\ref{LM:morphism-category},
        we may identify $\End(U_T,g_T,V_T)$ with $\End(U,g,T)\otimes_RT$
        for every flat $R$-algebra $T$. Recall that $A_S$
        is hereditary by Theorem~\ref{TH:local-hereditarity}.
        Now, taking $W=A_S$ (so that $W_F=A_K$),
        Lemma~\ref{LM:morphism-category}  implies
        that $(U_S,g_S,V_S)$ can be written as a direct sum
        $\bigoplus_{i=1}^s(U_i,g_i,V_i)$ in $\Mor(\rproj{A_S})$ such that
        $\End(U_i,g_i,V_i)\otimes_RF=\End(U_i,g_i,V_i)\otimes_SK$
        is local for all $i$.
        Let $e_i\in\End(U_S,g_S,V_S)=\End(U,g,V)\otimes_RS$
        be the projection onto the summand $(U_i,g_i,V_i)$.
        Then the idempotents $e_1,\dots,e_s$
        satisfy condition (2) of Theorem~\ref{TH:hereditary-orders}
        for $A'':=\End(U,g,V)$. That $A''_S$ is semilocal and $A''_K$ is semiperfect
        is clear, so we are done.
    \end{proof}

\subsection{Cancellation and Springer's Theorem}

    As an application of the previous results, we now give versions
    of Witt's Cancellation Theorem and Springer's (weak) Theorem.
    \emph{We assume $2\in\units{R}$ throughout}.

    \begin{cor}[Cancellation]\label{CR:cancellation}
        Let $A$ be a hereditary $R$-order,
        let $\sigma:A\to A$ be an $R$-involution and let $u\in\Cent{A}$ be an element in $A$
        satisfying $u^\sigma u=1$.
        Let $(P,f)$, $(P',f')$, $(P'',f'')$
        be $u$-hermitian forms over $(A,\sigma)$. If
        $f\perp f'\cong f\perp f''$, then $f'\cong f''$.
    \end{cor}

    \begin{proof}
        By Theorem~\ref{TH:non-unimodular-forms}, it is enough
        to show that $f'_{\hat{R}_\frakp}\cong f''_{\hat{R}_\frakp}$
        for all $\frakp\in\Spec R$ (including $\frakp=0$).
        Since $f_{\hat{R}_\frakp}\perp f'_{\hat{R}_\frakp}\cong f_{\hat{R}_\frakp}\perp f''_{\hat{R}_\frakp}$,
        it is enough to show that cancellation holds for (arbitrary) $u$-hermitian forms over
        finite $\hat{R}_\frakp$-algebras, and this was shown in \cite[Th.~3]{BayerFain96}
        (see also \cite[\S5A]{BayFiMol13} for further results of this kind).
    \end{proof}

    \begin{remark}
        Over general semilocal rings, cancellation fails even in the unimodular case; see \cite{Keller88}.
    \end{remark}

    For the next result, recall that an \emph{\'{e}tale extension} of $R$
    is a faithfully flat finitely presented commutative $R$-algebra $R'$ such that
    for every $\frakp\in\Spec R$ (including $\frakp=0$),
    $R'_{k(\frakp)}:=R'\otimes_R k(\frakp)$ is a finite product of separable field extensions of
    $k(\frakp)$ (where $k(\frakp)=R_\frakp/\frakp_\frakp$). This is equivalent
    to $\Spec R'\to \Spec R$ being an \emph{\'{e}tale cover}.
    For example, all {Galois extensions} of $R$ are \'{e}tale, and
    the algebra $\prod_{0\neq \frakp\in\Spec R}R_\frakp$ is also \'{e}tale over $R$.
    The rank of $R'$ is a function $\rank (R'):\Spec R\to \Z$ taking $\frakp$
    to $\dim_{k(\frakp)}R'_{k(\frakp)}$. This function is constant if $R'$ is a
    finitely generated projective $R$-module.
    We say that $R'$ has \emph{odd rank} over $R$ if $\rank(R')$ attains only odd values.

    \begin{cor}[Springer's Theorem]\label{CR:springer}
        Let $A,\sigma, u$ be as in Corollary~\ref{CR:cancellation}, and
        let $R'$ be an \'{e}tale $R$-algebra of odd rank.
        Let $(P,f),(P',f')$ be $u$-hermitian forms over $(A,\sigma)$. If
        $f_{R'}\cong f'_{R'}$, then $f\cong f'$.
    \end{cor}

    \begin{proof}
        It is easy to check that if $T$ is any commutative $R$-algebra,
        then $R'_T$ is \'{e}tale of odd rank over $T$.
        Thus, as in the proof
        of Corollary~\ref{CR:cancellation}, it is enough to prove the theorem
        in case $R$ is a complete discrete valuation ring.

        Now, as in the proof of Theorem~\ref{TH:non-unimodular-forms},
        we may reduce to \emph{unimodular} $1$-hermitian forms over an involutary $R$-order (not-necessarily
        hereditary), so assume henceforth $f$ and $f'$ are unimodular.
        Write $k=R/\Jac(R)$ and $k'=R'\otimes_R k$.
        Since $R$ is complete in the $\Jac(R)$-adic topology,
        it is enough to prove $f_k\cong f'_k$ (\cite[Th.~2.2(2)]{QuSchSch79} or \cite[Th.~II.4.6.1]{Kn91}).

        Indeed, $f_{R'}\cong f_{R'}$ implies $f_{k'}\cong f'_{k'}$.
        Since $k'$ is \'{e}tale over $k$, $k'$ is a finite product of field extensions of $k$.
        Since $k'$ has odd rank over $k$, at least one of these fields, denote it $k_0$, has odd dimension over $k$.
        Now, $f_{k_0}\cong f'_{k_0}$, so by \cite[Prp.~1.2, Th.~2.1]{BayLen90} (Springer's Theorem
        for quadratic spaces over finite dimensional involutary algebras), $f_k\cong f'_k$ and we are done.
    \end{proof}

    \begin{remark}
        Using similar ideas, one can prove Corollaries~\ref{CR:cancellation}
        and~\ref{CR:springer}  for systems of sesquilinear forms
        in the setting of Theorem~\ref{TH:second-kind-systems}.
    \end{remark}

\appendix{}

\section{Isometry Groups as Groups Schemes}

    This appendix shows that under certain assumptions, isometry
    groups can be realized as faithfully flat smooth affine group schemes. The
    arguments presented are almost entirely due to Mathieu Huruguen
    and we thank him for his contribution.

\medskip

    Throughout, $R$ is a commutative ring and
    $(A,\sigma,u,\Lambda)$ is a unitary $R$-algebra (see~\ref{subsection:unitary-algs}).
    We further assume that $A$ is finitely generated and projective over
    $R$ and $\Lambda$ is a summand of $A$.
    Recall that
    $\Comm{R}$ denotes the category of commutative $R$-algebras.
    The acronym ``fppf'' stands for ``finitely presented faithfully flat''.

    \begin{prp}
        Let $(P,[f]),(P',[f'])\in\Quad[u,\Lambda]{A,\sigma}$.
        Denote by $\Iso([f],[f'])$ the set of isometries from $(P,[f])$ to $(P',[f'])$.
        Then the functor
        \[S\mapsto \Iso([f_S],[f'_S]):\Comm{R}\to \mathrm{Set}\]
        is the functor of points of a smooth finitely presented affine scheme over $\Spec R$,
        denoted $\uIso([f],[f'])$.
        Furthermore,
        if for every   $\frakp\in\Spec R$, there is a field extension $L/k(\frakp)$
        such that
        $P_L\cong P'_L$ (as $A_L$-modules),
        then $\uIso([f],[f'])$ is faithfully flat over $\Spec R$, and there exists an fppf
        $R$-algebra $S$ such that
        such that $[f_S]\cong [f'_S]$.
    \end{prp}

    \begin{proof}
        Choose $B,Q,Q'\in\rproj{R}$ such that $A\oplus B$, $P\oplus Q$, $P'\oplus Q'$
        are free, and fix isomorphisms $A\oplus B\cong R^k$, $P\oplus Q\cong R^m$,
        $P'\oplus Q'\cong R^n$ (elements of $R^k$, $R^m$, $R^n$ are viewed as column vectors).
        For every $S\in \Comm{R}$, we embed $\Iso([f_S],[f'_S])$ in $\nMat{S}{n\times m}\times\nMat{S}{m\times n}\cong S^{2mn}$
        via
        \[\phi\mapsto (\phi\oplus 0,\phi^{-1}\oplus 0)\in\Hom_S(S^m,S^n)\times\Hom(S^n,S^m)\cong \nMat{S}{n\times m}\times\nMat{S}{m\times n}\ .\]
        We shall give finitely many polynomial equations on $R^{2mn}\cong\nMat{R}{m\times n}\times\nMat{R}{n\times m}$
        such that for every
        $S\in\Comm{R}$, their  zero locus
        in $S^{2mn}$ is   $\Iso([f_S],[f'_S])$. This will prove that the functor $S\mapsto \Iso([f_S],[f'_S])$
        is the points functor of the closed subscheme of $\mathbb{A}_R^{2mn}$ defined by these equations.

\smallskip

        Let $E\in\nMat{R}{m\times m}$ and $E'\in\nMat{R}{n\times n}$
        be the matrices corresponding to the projections $R^m=P\oplus Q\to P$ and
        $R^n=P'\oplus Q'\to P'$, respectively. Likewise, let
        $G\in\nMat{R}{k\times k}$ be a matrix corresponding
        to some projection of $R^k=A\oplus B$ onto the summand $\Lambda$.
        Let $a_1,\dots,a_r\in A$ be a set of generators of $A$ as an $R$-algebra.
        For each $1\leq i\leq r$, let $A_i\in\nMat{R}{m\times m}$ be the matrix corresponding to the map
        $x\oplus y\mapsto xa_i\oplus 0:P\oplus Q\to P\oplus Q$.
        Define $A'_1,\dots,A'_r\in\nMat{R}{n\times n}$ similarly, by replacing $P,Q$ with $P',Q'$.
        Next, extend $\tilde{h}_f:P\times P\to A$ to an $R$-bilinear function $H:R^m\times R^m\to R^k$
        by $H(x_1\oplus y_1,x_2\oplus y_2)=\tilde{h}_f(x_1,x_2)\oplus 0$ (where $x_1,x_2\in P$, $y_1,y_2\in Q$).
        There are matrices $H_1,\dots,H_k\in \nMat{R}{m\times m}$
        such that $H(x,y)=(x^{T}H_1y,\dots,x^TH_ky)$ for all $x,y\in R^m$.
        Similarly, define $H'_1,\dots,H'_k\in\nMat{R}{n\times n}$
        by replacing $h_f$ with $h_{f'}$. We repeat this procedure with
        $f$ and $f'$ in place of  $h_f$ and $h_{f'}$ to get matrices $F_1,\dots,F_k\in\nMat{R}{m\times m}$,
        $F'_1,\dots,F'_k\in\nMat{R}{n\times n}$. Finally, for a matrix $X$,
        denote by $d(X)$ the row vector consisting of the diagonal entries of $X$.
        It is now routine to check that a pair of matrices $(X,Y)\in \nMat{S}{m\times n}\times \nMat{S}{n\times m}$
        lies in $\Iso([f_S],[f'_S])$ if and only if the following matrix equations, which are defined
        over $R$, are satisfied:
        \begin{enumerate}
            \item $E'XE=X$
            (``there is
            $\phi\in\Hom_R(P, P')$ such that $X=\phi\oplus 0$''),
            \item $EYE'=Y$, $XY=E'$, $YX=E$  (``$\phi$ is invertible, $Y=\phi^{-1}\oplus 0$''),
            \item $XA_1=A'_1X, \dots, XA_r=A'_rX$ (``$\phi$ is $A$-linear''),
            \item $X^TH'_1X=H_1, \dots, X^TH'_kX=H_k$ (``$\tilde{h}_{f'}(\phi x,\phi y)=\tilde{h}_f(x,y)$''),
            \item $G\cdot D(X)=D(X)$, where  $D(X)$ is the matrix whose rows are\linebreak
            $d(X^TF'_1X-F_1),\dots,d(X^TF'_kX-F_k)$ (``$\tilde{f}'(\phi x,\phi x)-\tilde{f}(x,x)\in \Lambda$'').
        \end{enumerate}

\smallskip

        We now show that $\uIso([f],[f'])$ is smooth over $\Spec R$. Since $\uIso([f],[f'])$ is of finite
        presentation, it is enough to check that $\uIso([f],[f'])$ is \emph{formally smooth}
        (see~\cite[Th.~17.5.1]{Groth67EGAiv}).
        Let $S$ be a commutative $R$-algebra and let $I\lhd S$ be an ideal with $I^2=0$.
        We need to show that
        \[
        \xi\mapsto \quo{\xi}:\Iso([f_S],[f'_S])\to \Iso([f_{S/I}],[f'_{S/I}])
        \]
        is surjective. This is a well-known argument (e.g.\ see \cite[Th.~2.2]{QuSchSch79}); we
        recall it for the sake of completeness.
        Let $\phi\in \Iso([f_{S/I}],[f'_{S/I}])$.
        Since $P$ and $P'$ are projective over $A$, taking $\Hom$-sets
        and dualizing with respect to $A$ commute with scalar extension (up to natural
        isomorphism; cf.\ Lemma~\ref{LM:natural-isomorphism}).
        Thus, there exists
        $\psi\in\Hom_{A_S}(P_S,P'_S)\cong \Hom_A(P,P')_S$ whose image
        in $\Hom_{A_{S/I}}(P_{S/I},P'_{S/I})\cong \Hom_A(P,P')_{S/I}$ is $\phi$.
        Since $\phi$ is an isometry, we have
        $\phi^*f'_{S/I}\phi-f_{S/I}\in \Lambda_{P_{S/I}}$.
        By Proposition~\ref{PR:scalar-ext-for-quad-forms}, the map $\Lambda_{P_S}\otimes_S S/I\to\Lambda_{P_{S/I}}$ is surjective,
        and hence there is $r\in \Lambda_{P_S}$
        such that
        \[
        g:=\psi^*f'_S\psi-f_S-r\in\ker\big(\Hom_A(P,P')_S\to \Hom_A(P,P')_{S/I}\big)=I\cdot \Hom_A(P,P')_S\ .
        \]
        Let  $c:=-(h_{f'_S})^{-1}(\psi^*)^{-1}g$ and $\xi=\psi+c$
        (note that $h_{f'_S}=f'_S+f'^*_S\omega$ is invertible because $[f']$ is unimodular).
        Then $\quo{\xi}=\phi$ (because $c\in I\cdot\Hom_A(P,P^*)_S$), and it straightforward
        to check that
        \begin{align*}
        \xi^*f'_S\xi-f_S
        & =r+
        (c^*f'_S\psi)-(c^*f'_S\psi)^*\omega\in\Lambda_{P_S}
        \end{align*}
        (use the fact that $g=-\psi^*(f'_S+f'^*_S\omega)c$).
        Thus, $\xi$ is an isometry from $[f_S]$ to $[f'_S]$.
        That $\xi$ is invertible
        is shown by similar means  and is left to the reader.

\smallskip

        To finish, assume that for every $\frakp\in \Spec R$, there is a field
        extension $L/k(\frakp)$ with $P_L\cong P'_L$,
        and write $\uIso([f],[f'])=\Spec S$ for suitable $S\in\Comm{R}$.
        Since $\uIso([f],[f'])(S)=\Iso([f_S],[f'_S])$ is tautologically
        non-empty, it is enough to show that $\Spec S\to \Spec R$ is faithfully flat.
        Since this morphism is smooth, and in particular flat, it is left to check that
        $\Spec S\to \Spec R$ has non-empty fibers.
        Indeed, let $\frakp\in\Spec R$, and choose an \emph{algebraically
        closed} field $L\supseteq k(\frakp)$ such that $P_L\cong P'_L$. By \cite[3.4(3)]{QuSchSch79} (for instance),
        the latter implies $[f_L]\cong [f'_L]$, so $\uIso([f],[f'])(L)\neq \emptyset$.
        Thus, the fiber $\Spec S\times_{\Spec R} \Spec k(\frakp)$ has an $L$-point.
    \end{proof}

    For $(P,[f])\in\Quad[u,\Lambda]{A,\sigma}$, we write $\uO([f])=\uIso([f],[f])$.
    The proposition shows that  $\uO([f])$ is a
    smooth fppf affine scheme over $\Spec R$.

    \begin{remark}\label{RM:appnedix-remark}
        (i) The proof of the proposition also shows that isometry groups of \emph{non-unimodular}
        forms can be regarded as affine schemes over $\Spec R$, but they are not flat in general.

        (ii) The scheme $\uO([f])$
        is a closed subscheme of $\underline{\End_A(P)}$, the scheme corresponding to $\End_A(P)$ (cf.~Proposition~\ref{PR:realizing-projective-mods}):
        Keeping the notation of the proof, embed $\Iso([f_S],[f_S])$ in $\End_{S}(P\oplus Q)\cong \nMat{S}{m}\cong S^{m^2}$ via
        $\phi\mapsto \phi\oplus 0$.
        Then $\End_{A_S}(P_S)$ is the zero locus of  equations (1) and (3) above,
        while  $\Iso([f_S],[f_S])$ is the zero locus of equations (1) and (3)--(5), hence our claim.
        (Note that equation (4) implies that the isometry $\phi$ corresponding to $X$
        is invertible: Since $\phi^* h_f\phi=h_f$ and $[f]$ is unimodular,
        $\phi$ is invertible on the left,
        and hence so is $\phi\oplus\id_Q\in\End(P\oplus Q)\cong \nMat{S}{m}$.
        Now, easy determinant considerations imply $\phi\oplus\id_Q$ is invertible on the right, so $\phi$ is invertible.)
    \end{remark}

\bibliographystyle{plain}
\bibliography{MyBib_15_09}

\end{document}